\documentclass[a4paper,10pt]{article}
\usepackage{hyperref}
\usepackage[T1]{fontenc}
\usepackage[utf8]{inputenc}

\usepackage{natbib}

\usepackage{todonotes}

\usepackage[english]{babel}
\usepackage{mathtools}
\usepackage{amsfonts}
\usepackage{color}
\usepackage{amsmath}
\usepackage{amsthm}
\usepackage{amssymb}
\usepackage{mathrsfs}
\usepackage{amstext}

\usepackage[normalem]{ulem}

\usepackage{graphicx}
\usepackage{tikz}
\usepackage{subfig}
\usepackage{pgfplots}
\usepgfplotslibrary{fillbetween}
\pgfplotsset{compat=1.18}
\usepackage{subcaption}

\usepackage{authblk}
\usepackage{stmaryrd}
\usepackage{makeidx}
\makeindex

\numberwithin{equation}{section}

\theoremstyle{definition}
\newtheorem{thm}{Theorem}[section]
\newtheorem{rem}[thm]{Remark}
\newtheorem{defi}[thm]{Definition}
\newtheorem{lem}[thm]{Lemma}
\newtheorem{cor}[thm]{Corollary}
\newtheorem{prop}[thm]{Proposition}

\newtheorem*{thm*}{Theorem}
\newtheorem*{rem*}{Remark}
\newtheorem*{folg*}{Folgerung}
\newtheorem*{examples*}{Beispiele}
\newtheorem*{ex*}{Beispiel}
\newtheorem*{lem*}{Lemma}
\newtheorem*{prop*}{Proposition}
\newtheorem*{defi*}{Definition}
\newtheorem*{exercise*}{Übung}
\newtheorem*{conj*}{Conjecture}
\newtheorem*{q*}{Question}

\usepackage{cleveref}

\newcommand{\bB}{\mathbf{B}}

\newcommand{\bE}{\mathbf{E}}

\newcommand{\bl}{\underline{\ell}}
\newcommand{\bbeta}{\underline{\beta}}
\newcommand{\bgamma}{\underline{\gamma}}

\newcommand{\bj}{\mathbf{j}}
\newcommand{\bk}{\mathbf{k}}
\newcommand{\bm}{\mathbf{m}}
\newcommand{\bn}{\mathbf{n}}
\newcommand{\bs}{\mathbf{s}}

\newcommand{\bv}{\mathbf{v}}

\newcommand{\bx}{\mathbf{x}}
\newcommand{\by}{\mathbf{y}}

\newcommand{\bbC}{\mathbb{C}}
\newcommand{\bbD}{\mathbb{D}}

\newcommand{\bbN}{\mathbb{N}}

\newcommand{\bbR}{\mathbb{R}}

\newcommand{\bbT}{\mathbb{T}}

\newcommand{\bbZ}{\mathbb{Z}}

\newcommand{\cF}{\mathcal{F}}

\newcommand{\cL}{\mathcal{L}}

\newcommand{\cR}{\mathcal{R}}

\newcommand{\diffd}{\mathrm{d}}

\newcommand{\es}{\emptyset}

\newcommand{\iu}{\mathrm{i}}

\newcommand{\sbse}{\subseteq}
\newcommand{\sbsn}{\subsetneq}

\newcommand{\QMC}{e^{\text{QMC}}}
\newcommand{\sinc}{\operatorname{sinc}}
\newcommand{\supp}{\operatorname{supp}}

\newcommand{\Hmix}{H_\text{mix}}

\newcommand{\seqsp}{H_{\text{mix}}^{2,\text{dyad}}}
\newcommand{\Fab}{\mathsf{Fab}}

\author{Bernd K\"a{\ss}emodel\thanks{bernd.kaessemodel@mathematik.tu-chemnitz.de}, Nicolas Nagel\thanks{nicolas.nagel@mathematik.tu-chemnitz.de}, Tino Ullrich\thanks{tino.ullrich@mathematik.tu-chemnitz.de}}
\affil{\small{Technische Universität Chemnitz\\
		Fakultät für Mathematik\\
		09107 Chemnitz }}

\title{Tent transformed order $2$ nets and quasi-Monte Carlo rules with quadratic error decay}
\date{}

\pgfplotsset{yticklabel style={text width=3em,align=right}}

\begin{document}
	
	\maketitle
	
	
	\begin{abstract}
		We investigate the use of order $2$ digital nets for quasi-Monte Carlo quadrature of nonperiodic functions with bounded mixed second derivative over the cube. By using the so-called tent transform and its mapping properties we inherit error bounds from the periodic setting. Our analysis is based on decay properties of the multivariate Faber-Schauder coefficients of functions with bounded mixed second weak derivatives. As already observed by Hinrichs, Markhasin, Oettershagen, T. Ullrich (Numerische Mathematik 2016),  order $2$ nets  work particularly well on tensorized (periodic) Faber splines. From this we obtain a quadratic decay rate for tent transformed order $2$ nets also in the nonperiodic setting. This improves the formerly best known bound for this class of point sets by a factor of $\log N$.
		
		We back up our findings with numerical experiments, even suggesting that the bounds for order $2$ nets can be improved even further. This particularly indicates that point sets of lower complexity (compared to previously considered constructions) may already give (near) optimal error decay rates for quadrature of functions with second order mixed smoothness. 
	\end{abstract}
	
	\section{Introduction}
	
	\subsection{The quasi-Monte Carlo rule}
	High-dimensional quadrature is an important research field in numerical analysis and approximation theory. There are various methods and algorithms for this task but the simplest one is arguably the \emph{quasi-Monte Carlo integration rule}. Given a domain $\Omega$ together with a normalized measure $\mu$ (that is $\mu(\Omega) = 1$) we aim to approximate
	$$
	\int_\Omega f(x) \, \diffd \mu(x) \approx \frac1N \sum_{x \in X} f(x)
	$$
	where $X \sbse \Omega, \# X = N$ is a finite set of evaluation nodes and $f$ is a real-valued function over $\Omega$, usually fulfilling some smoothness assumptions. The relative worst case error for this algorithm measures the quality of the point set $X$ in $\Omega$.
	
	In this paper we focus on the case where $\Omega = [0, 1]^d$ is the $d$-dimensional unit cube and $f$ is a (not necessarily periodic) function with dominating mixed second order smoothness. Let $\Hmix^2([0, 1]^d)$ be the Sobolev space of bounded mixed smoothness $2$ as will be defined in Section \ref{subsec:Sobolev}. The \emph{worst case error} of the quasi-Monte Carlo quadrature rule of a finite set $X \sbse [0, 1]^d, \#X = N$ over this space is given by
	$$
	\QMC(X, \Hmix^2([0, 1]^d)) \coloneqq \sup\limits_{\|f\|_{\Hmix^2([0, 1]^d)} \leq 1}\left|\frac1N \sum_{\bx \in X} f(\bx) - \int_{[0, 1]^d} f(\bx) \, \diffd\bx\right|
	$$
	and the \emph{optimal worst case error} for a given number of points $N$ is defined by
	$$
	\QMC(N, \Hmix^2([0, 1]^d)) \coloneqq \inf\limits_{\substack{X \sbse [0, 1]^d \\ \# X = N}} \QMC(X, \Hmix^2([0, 1]^d)).
	$$
	The asymptotic behavior of this quantity over certain function spaces as $N$ goes to infinity is a popular and well-studied problem. The classical result in this direction is the Koksma-Hlawka inequality \cite{Hic2014, Hla1984, KN1974} where the norm is given by the Hardy-Krause variation. In this case the worst case error turns out to be the star discrepancy of the point set $X$. This is a general theme and in many cases the worst case error possesses such a geometric interpretation via discrepancy. For more details and asymptotic results on discrepancy see \cite{BLV2008, DP2010, Nie1992}.
	
	\subsection{Quasi-Monte Carlo via order $2$ nets}
	
	Our setting is strongly inspired by \cite{Dic2007, HMOU2016} where the analogous setting for periodic functions (that is where the domain is the torus $\bbT^d$) was investigated. Upper bounds via constructive point sets were considered in \cite{Dic2008} (building on \cite{Dic2007}), where the bound
	\begin{align} \label{ineq:Dick}
		\QMC(X, \Hmix^s([0, 1]^d)) \lesssim \frac{(\log N)^{ds}}{N^s}
	\end{align}
	via order $s$ digital nets $X$ was determined. Digital nets use bit-manipulation techniques to generate point sets that distribute themselves well in the cube (in this paper we will only consider nets in base $b=2$). The order parameter indicates a further modification, where we compress high-dimensional nets via digital interlacing to construct point sets with a fine structure that adapts itself to higher smoothness parameters $s$. Details will be discussed in Section \ref{sec:nets}. Later on, it was shown in \cite{BD2009} that one can achieve
	\begin{align*} 
		\QMC(X, \Hmix^s([0, 1]^d)) \lesssim \frac{(\log N)^{ds/2}}{N^s}
	\end{align*}
	via randomly shifted order $s$ nets. However, this makes the point set probabilistic and nonconstructive.
	
	The optimal behavior of the worst case error was determined in \cite{GSY16, GSY2016, GSY2018}. Indeed, for integer smoothness $s \geq 2$ it was shown that
	\begin{align} \label{eq:GSY_optimal_order}
		\QMC(X, \Hmix^s([0, 1]^d)) \lesssim \frac{(\log N)^{(d-1)/2}}{N^s}
	\end{align}
	for order $2s+1$ nets in base $2$ (order $2s$ if one also considers sufficiently large bases $b > 2$), giving an explicit construction. For smoothness $s=2$ one thus needs order $5$ binary nets according to this method.
	
	For another, randomized setting see \cite{Dic2011} and for recent results concerning lattice constructions see \cite{GK2024}.
	
	In this paper we will consider a simple modification to order $2$ digital nets coming from a periodization procedure called the \emph{tent transform}. For these constructive point sets we will prove, for smoothness $s=2$, a rate of
	\begin{equation}\label{eq:bound}
		\QMC(X, \Hmix^2([0, 1]^d)) \lesssim \frac{(\log N)^{2d-1}}{N^2},
	\end{equation}
	see Theorem \ref{thm:upper_bound}. This gives an improvement over \eqref{ineq:Dick} for the same order parameter $2$ and shows that \eqref{ineq:Dick} is not best possible for constructions based on order $2$ nets. Numerical experiments suggest that this bound might even be further improved, see Section \ref{sec:numerics}. The utility of the tent transform for quasi-Monte Carlo integration with digital nets was already observed in \cite{CDLP07}, although the author there could not determine a logarithmic factor.
	
	Reducing the order of the constructed point sets is of interest for two reasons. First, from a practical point of view, it is easier to construct lower order digital nets than higher order ones. This is due to the construction described in Section \ref{sec:nets}, where one needs to start with an $\alpha d$-dimensional digital net to get a point set in $d$ dimensions. Since computer implementations of these point sets are highly nontrivial and are usually only available up to dimensions $33$ \cite{KN2016}, lowering the order parameter is highly required. For example, using this data base one can generate order $2$ nets up to $d=16$ dimensions, but order $5$ nets only up to $d=6$ dimensions.
	
	A second, more theoretical aspect is the question of how complex a point set must be to achieve (asymptotically) the optimal rate. This can be made precise as follows. The point sets considered in \cite{GSY2018} for smoothness $s=2$ consist of nodes whose coordinates use $5 \log_2 N$ bits in their binary representation. The tent transformed order $2$ nets only use $2 \log_2 N - 1$ bits in the representation of their components. It is thus natural to ask how many bits are necessary to obtain point sets of asymptotically optimal behavior. In general, order $\alpha$ digital nets are, in the sense as just described, of complexity $\alpha \log_2 N$. The question of the required complexity is thus related to the smallest order parameter necessary (depending on the smoothness parameter) for optimal bounds of the worst case error. We are not aware if this aspect has been studied in the existing literature, although it is similar to \cite{HO2016}, where point sets which can be well approximated by permutations have been determined to be globally optimal for the periodic case in dimension $d=2$ and smoothness $s=1$. This can of course be considered, also using other notions of complexity, for other problems in discrepancy and quasi-Monte Carlo theory.
	
	\subsection{Faber-Schauder coefficients and tent transform}
	
	For functions $f \in \Hmix^2(\bbR^d)$ possessing second order mixed weak derivatives bounded in $L_2(\bbR^d)$ we rigorously prove the following decay property of the Faber-Schauder coefficients $d_{\bj,\bk}(f)$, $\bj \in \bbN_{-1}^d, \bk \in \bbZ^d$, see Section \ref{sec:tfaber}. It holds
	\begin{equation}\label{eq:18}
		\sup\limits_{\bj \in \bbN_{-1}^d} 2^{ 3|\bj|_1/2} \left(\sum_{\bk \in \bbZ^d} |d_{\bj, \bk}(f)|^2\right)^{1/2} 
		\lesssim \|f\|_{\Hmix^2(\bbR^d)}.
	\end{equation}
	This is our main tool for proving the error bound \eqref{eq:bound}. The above relation directly implies the corresponding relation for spaces $\Hmix^2([0,1]^d)$. The left-hand side does not seem to fit properly to right-hand side if one has in mind the corresponding relation for $1/2<s<2$, see \cite{HMOU2016}. However, the results in \cite[Theorem 1.8]{GSU23} indicate that the left-hand side might even be an equivalent norm on the space $\Hmix^2$. We leave this as an open problem.  
	
	We further use the boundedness of the tent transform operator 
	$$
	\mathcal{R}:f(x_1, \dots, x_d) \mapsto f(|2x_1-1|,\dots,|2x_d-1|) \quad (\bx \in [0,1]^d)\,,
	$$
	in various situations, see Section \ref{sec:tent}. Namely, we have 
	\begin{equation}\label{eq:emb}
		\begin{split}
			\cR&: \Hmix^2([0, 1]^d) \rightarrow \seqsp(\bbT^d) \quad \mbox{(Theorem \ref{thm:FS_charact.} and Theorem \ref{thm:tent})}\\
			\cR&: \Hmix^2([0, 1]^d) \rightarrow \bE^2_d\quad \mbox{(Corollary \ref{cor:Korobov})} \\
			\cR&: \bB_{1,\infty}^{2,\text{dyad}}([0,1]^d) \to \bB_{1,\infty}^{2,\text{dyad}}(\bbT^d). 
		\end{split}
	\end{equation}
	
	The last two embeddings are relevant due to the fact that several applications, for example in mathematical finance, require the treatment and quadrature of sums of tensor products of piecewise linear functions. Clearly, these functions neither belong to $\Hmix^2$ nor to $\Hmix^{2,\text{dyad}}$. Numerical experiments show (see Figure \ref{fig:piecewise_linear} below) that the errors decay with main rate $N^{-2}$ or even better. There are several ways to explain this effect. One is that the tent transform maps such functions into the space $\mathbf{E}^2_d$, the classical Korobov space, see \cite[Section 3.3]{DTU18}, where one knows the corresponding decay rates \cite{Dic2007}. This direction will be elaborated on in Section \ref{sec:Korobov}. Here we want to point out a different reason, namely the third embedding in \eqref{eq:emb}, where for $1\leq p<\infty$
	$$
	\|f\|_{\bB_{p,\infty}^{2,\text{dyad}}}\coloneqq
	\sup\limits_{\bj \in \bbN_{-1}^d} 2^{ |\bj|_1(2-1/p)} \left(\sum_{\bk} |d_{\bj, \bk}(f)|^p\right)^{1/p} < \infty. 
	$$
	Note, that in this notation we have the identification $\bB_{2,\infty}^{2,\text{dyad}} = \Hmix^{2,\text{dyad}}$. The range of $\bk$ in the sum above depends on which domain ($[0,1]^d$ or $\bbT^d$) is used, see Definition \ref{def:univar_Faber}, its higher dimensional analogue and Definition \ref{def:seq_space} below. The mentioned test functions belong to such spaces, where we trade $L_2$-integrability against $L_1$-integrability and gain smoothness. Here we use the dyadic versions of corresponding Besov spaces with dominating mixed smoothness, similar as in \cite[Section 1.3]{GSU23}. As a direct consequence we establish an analogous bound to \eqref{eq:bound}, however, now for the significantly larger space $\bB_{1,\infty}^{2,\text{dyad}}([0,1]^d)$, see \eqref{eq:21}.
	
	\paragraph{Notation.}
	The $d$-dimensional torus will be denoted by $\bbT^d$ and its points will be identified with points in $[0, 1)^d$. Vectors will be denoted in bold type $\bk, \bx$ and so on (with components $k_i, x_i$ and so on), with the exceptions $\bl, \underline{\beta}, \underline{\gamma} \in \bbR^d$ (with components $\ell_i, \beta_i$ and so on). We will use $\bbN_0 = \{0, 1, 2, \dots, \}$ and $\bbN_{-1} = \{-1\} \cup \bbN_0$. For $\bj \in \bbN_{-1}^d$ we set $|\bj|_1 = |j_1|+\dots+|j_d|$. The positive part of a number $x \in \bbR$ will be denoted by $(x)_+ \coloneqq \max\{x, 0\}$. The indicator function of a set $B$ will be denoted by $\chi_B^{}$. We will denote (weak) partial derivatives by
	$$
	f^{(s_1, \dots, s_d)}(\bx) \coloneqq \left(\prod_{i = 1}^d \frac{\partial^{s_i}}{\partial x_i^{s_i}}\right) f(\bx).
	$$
	The \emph{Schwartz space} of rapidly decreasing functions will be denoted by $S(\bbR^d)$. For $f \in S(\bbR^d)$ the Fourier transform over $\bbR^d$ is defined as
	$$
	\cF_d [f](\bs) \coloneqq \int_{\bbR^d} f(\bx) \exp(-2\pi\iu \bs^\top \bx) \, \diffd \bx
	$$
	for $\bs \in \bbR^d$, with inverse given by $\cF_d^{-1} [f](\bs)\coloneqq\cF_d[f](-\bs)$. The class of bounded, continuous functions over a domain $\Omega$ will be denoted by $C_b(\Omega)$ and equipped with the $\sup$-norm $\|\cdot\|_\infty$. For functions $f$ and $g$ we write $f \lesssim g$ if there is a constant $C > 0$ such that $f \leq C g$ for all inputs. We will use $f \asymp g$ as an abbreviation for $f \lesssim g$ and $g \lesssim f$.
	
	\section{Sobolev spaces and their tensor products} \label{subsec:Sobolev}
	
	The univariate Sobolev space $H^2([0, 1])$ of smoothness $2$ can be defined by
	$$
	H^2([0, 1]) = \{f \in L_2([0, 1]): f'' \in L_2([0, 1])\},
	$$
	where $f''=f^{(2)}$ denotes the second weak derivative. We will only consider real valued functions $f$. There are multiple, equivalent norms we can define on this space as to turn it into a Hilbert space. It is common to use
	$$
	\|f\|_{H^2([0, 1])}^2 \coloneqq \|f\|_{L_2([0, 1])}^2 + \|f''\|_{L_2([0, 1])}^2, 
	$$
	also giving the space an obvious inner product turning it into a Hilbert space. From \cite[Examples 21-23]{BTA2003} we will also consider
	\begin{align*}
		\|f\|_1^2 & \coloneqq \left(\int_0^1 f(x) \, \diffd x\right)^2 + \left(\int_0^1 f'(x) \, \diffd x\right)^2 + \int_0^1 f''(x)^2 \, \diffd x, \\
		\|f\|_2^2 & \coloneqq f(0)^2 + f(1)^2 + \int_0^1 f''(x)^2 \, \diffd x, \\
		\|f\|_3^2 & \coloneqq f(0)^2 + f'(0)^2 + \int_0^1 f''(x)^2 \, \diffd x.
	\end{align*}
	These norms even turn out to give $H^2([0, 1])$ the structure of a \emph{reproducing kernel Hilbert space} (RKHS), see \cite{BTA2003} for details. For $\|\cdot\|_j, j=1, 2, 3$, the kernels can be given in closed form by
	\begin{align*}
		K^1_1(x, y) & = 1 + B_1(x) B_1(y) + \frac14 B_2(x) B_2(y) - \frac1{24} B_4(|x-y|), \\
		K_2^1(x, y) & = 1 - x - y + 2xy + \frac16 (x-y)_+^3 - \frac16 x(1-y)\left(x^2-2y+y^2\right), \\
		K_3^1(x, y) & = 1 + xy + \frac13 \min\{x, y\}^3 - \frac12 (x+y) \min\{x, y\}^2 + xy \min\{x, y\},
	\end{align*}
	respectively, where
	$$
	B_1(x) = x-\frac12, \quad B_2(x) = x^2-x+\frac16, \quad B_4(x) = x^4-2x^3+x^2-\frac1{30}
	$$
	are Bernoulli polynomials. The $d$-variate Sobolev space with bounded mixed second derivatives $\Hmix^2([0, 1]^d)$ is given by the $d$-fold tensor product of $H^2([0, 1])$ with itself. We again obtain RKHSs with kernels given by
	$$
	K_j^d(\bx, \by) = \prod_{i=1}^d K_j^1(x_i, y_i)
	$$
	for $j=1, 2, 3$ respectively. The theory of numerical integration over RKHSs is especially nice since we can express the worst case error explicitly by \cite{NW2010}
	\begin{align*}
		& \sup\limits_{\|f\|_j \leq 1} \left|\int_{[0, 1]^d} f(\bx) \, \diffd\bx - \frac1N \sum_{\bx \in X} f(\bx) \right|^2 \\
		= & \iint_{[0, 1]^d} K_j^d(\bx, \by) \, \diffd (\bx, \by) - \frac2N \sum_{\bx \in X} \int_{[0, 1]^d} K_j^d(\bx, \by) \, \diffd \bx + \frac1{N^2} \sum_{\bx, \by \in X} K_j^d(\bx, \by),
	\end{align*}
	which can be given for $j = 1$ by
	\begin{align*}
		-1 + \frac1{N^2} \sum_{\bx, \by \in X} K_1^d(\bx, \by),
	\end{align*}
	for $j=2$ by
	\begin{align*}
		\left(\frac{61}{120}\right)^d - \frac2N \sum_{\bx \in X} \prod_{i=1}^d \left(\frac12 + \frac{x_i}{24} - \frac{x_i^3}{12} + \frac{x_i^4}{24}\right) + \frac1{N^2} \sum_{\bx, \by \in X} K_2^d(\bx, \by) 
	\end{align*}
	and for $j=3$ by
	\begin{align*}
		\left(\frac{13}{10}\right)^d - \frac2N \sum_{\bx \in X} \prod_{i=1}^d \left(1 + \frac{x_i}{2} + \frac{x_i^2}4 - \frac{x_i^3}{6} + \frac{x_i^4}{24}\right) + \frac1{N^2} \sum_{\bx, \by \in X} K_3^d(\bx, \by). 
	\end{align*}
	Given a point set $X$ it is thus easy to calculate the worst case error with respect to a given norm as the square root of these expressions.
	
	For later use we will also introduce the corresponding space over $\bbR^d$. The Sobolev space $\Hmix^2(\bbR^d)$ is defined analogously as the set of all functions $f: \bbR^d \rightarrow \bbR$ such that for all $e \sbse [d] \coloneqq \{1, \dots, d\}$ the weak partial derivatives
	$$
	f_e(\bx) \coloneqq \left(\prod_{i \in e} \frac{\partial^2}{\partial x_i^2}\right) f(\bx)
	$$
	are in $L_2(\bbR^d)$ (where $f_\es = f$). The norm is given by
	\begin{equation}\label{eq:10}
		\|f\|^2_{\Hmix^2(\bbR^d)} \coloneqq \sum_{e \sbse [d]} \|f_e\|^2_{L_2(\bbR^d)}
		\asymp \int_{\bbR^d} |\cF_d[f](\bs)|^2\prod\limits_{i=1}^d(1+|s_i|^2)^2\,\diffd\bs.
	\end{equation}
	
	Note, that we have in the sense of equivalent norms
	\begin{align}\label{eq:domain}
		\|f\|_{\Hmix^2([0, 1]^d)} \asymp \inf_{\substack{g \in \Hmix^2(\bbR^d) \\ g|_{[0, 1]^d} = f}} \|g\|_{\Hmix^2(\bbR^d)},
	\end{align}
	see \cite[Definition 1.56]{Tr10} with $\Omega = [0,1]^d$. In the mentioned reference $\Omega$ is supposed to be an open set. This does not play a role here since all appearing functions are continuous and therefore determined on the boundary.

	\section{The Faber basis} \label{HaarFaber}
	
	\subsection{The univariate Faber basis}
	
	\begin{figure}\label{fig:3}
		\centering
		\begin{tikzpicture}[scale=2.3]
			
			\draw[->] (-0.1,0.0) -- (1.1,0.0);
			\draw[->] (0.0,-0.1) -- (0.0,1.1); 
			
			\draw (1.0,0.03) -- (1.0,-0.03) node [below] {$1$};
			\draw (0.03,1.0) -- (-0.03,1.00) node [left] {$1$};
			\draw[->] (0.8,0.9) -- (0.7,0.7);
			
			\node at (0.8,1) {$v_{0,0}$};
			
			\node at (1.1,0.5) {$j=0$};
			
			\draw (0,0) -- (0.5,1);
			\draw (0.5,1) -- (1,0);
			
			\draw[->] (1.4,0.0) -- (2.6,0.0); 
			\draw[->] (1.5,-0.1) -- (1.5,1.1); 
			
			\draw (2.5,0.03) -- (2.5,-0.03) node [below] {$1$};
			\draw (1.53,1.0) -- (1.47,1.00) node [left] {$1$};
			\draw[->] (2.5,0.9) -- (2.4,0.7) ;
			\draw[->] (2,0.9) -- (1.9,0.7) ;
			
			\node at (2.5,1) {$v_{1,1}$};
			\node at (2,1) {$v_{1,0}$};
			
			\node at (2.6,0.5) {$j=1$};
			
			\draw (1.5,0,0) -- (1.75,1.0);
			\draw (1.75,1) -- (2,0.0);
			\draw (2,0) -- (2.25,1.0);
			\draw (2.25,1) -- (2.5,0.0);

			\draw[->] (2.9,0.0) -- (4.1,0.0); 
			\draw[->] (3.0,-0.1) -- (3.0,1.1);
			
			\draw (4.0,0.03) -- (4.0,-0.03) node [below] {$1$};
			\draw (3.03,1.0) -- (2.97,1.00) node [left] {$1$};
			
			\fill[gray!60] plot[domain=3:3.5] (\x,-6+2*\x)%
			-- plot[domain=4:3.5] (\x,0);%

			\node at (4.3,0.5) {$j\in\{0,1\}$};

			\draw (3,0) -- (3.5,1);
			\draw (3.5,1) -- (4,0);

			\draw (3.0,0,0) -- (3.25,1.0);
			\draw (3.25,1) -- (3.5,0.0);
			\draw (3.5,0) -- (3.75,1.0);
			\draw (3.75,1) -- (4,0.0);
			
		\end{tikzpicture}
		\caption{Univariate hierarchical Faber functions on $[0,1]$ for levels $j\in \{0,1\}$ and their union.} \label{fig_Faber1}
	\end{figure}
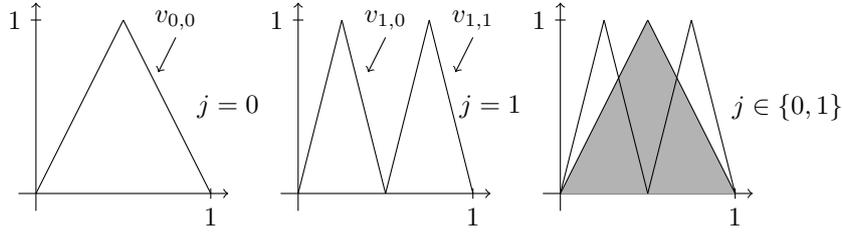
	
	Let us briefly recall the basic facts about the Faber basis taken from
	\cite[3.2.1, 3.2.2]{Tr10}. Faber \cite{Fa09} observed that every continuous
	(nonperiodic) function $f$ on $[0,1]$ can be represented (pointwise) as
	\begin{equation}\label{f51}
		f(x) = f(0)\cdot (1-x)+f(1)\cdot x -
		\frac{1}{2}\sum\limits_{j=0}^{\infty}\sum\limits_{k=0}^{2^j-1}
		\Delta^2_{2^{-j-1}}f(2^{-j}k)v_{j,k}(x),
	\end{equation}
	where
	\begin{equation}\label{eq:4}
		\Delta^2_hf(t) := f(t+2h)-2f(t+h)+f(t) \quad  (t,h\in \mathbb{R})
	\end{equation}
	denotes the second order difference operator at point $t$ with steplength $h$ and
	\begin{align*}
		v_{j, k}(x) \coloneqq \begin{cases}
			2^{j+1}(x-2^{-j}k) & , 2^{-j}k \leq x \leq 2^{-j}(k+1/2) \\
			2^{j+1}(2^{-j}(k+1)-x) & , 2^{-j}(k+1/2)\leq x \leq 2^{-j}(k+1) \\
			0 & , \text{else}
		\end{cases}
	\end{align*}
	for $j \geq 0$. The representation \eqref{f51} converges at least pointwise. In particular, every periodic function on
	$C(\bbT)$ can be represented by
	\begin{equation}\label{f5}
		f(x) = f(0) - \frac{1}{2}\sum\limits_{j=0}^{\infty}\sum\limits_{k=0}^{2^j-1}
		\Delta^2_{2^{-j-1}}f(2^{-j}k)v_{j,k}(x).
	\end{equation}
	\begin{defi} \label{def:univar_Faber}
		The \emph{univariate Faber basis over $\bbR$, $[0, 1]$, $\bbT$} is given, respectively, by
		\begin{align*}
			\Fab(\bbR) & \coloneqq \{v_{j, k}: j \in \bbN_{-1}, k \in \bbZ\}, \\
			\Fab({[0, 1]}) & \coloneqq \{v_{j, k}: j \in \bbN_{-1}, k \in \bbD_j\}, \\
			\Fab({\bbT}) & \coloneqq \{v_{j, k}^{\text{per}}: j \in \bbN_{-1}, k \in \bbD_j^{\text{per}}\},
		\end{align*}
		where we define the index sets $\bbD_j \coloneqq \bbD_j^{\text{per}} \coloneqq \{0, 1, \dots, 2^j-1\}$ for $j \geq 0$, as well as $\bbD_{-1} \coloneqq \{0, 1\}$ and $\bbD_{-1}^{\text{per}} \coloneqq \{0\}$. For $j \geq 0, k \in \bbD_j$ we set $v_{j, k}^{\text{per}} = v_{j, k}$ as defined above and for $j = -1$ we set
		$$
		v_{-1, k}(x) \coloneqq (1-|x-k|)_+ \quad (k \in \bbZ)
		$$
		and
		$$
		v_{-1, 0}^{\text{per}}(x) \coloneqq 1.
		$$
	\end{defi}
	
	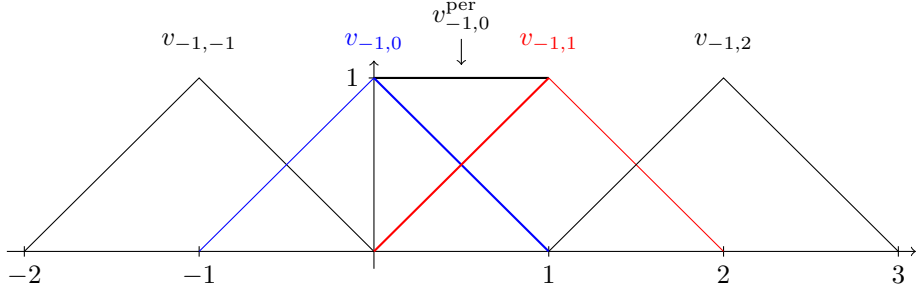
\begin{figure}
		\centering
		\begin{tikzpicture}[scale=2.3]
			
			\draw[->] (-2.1,0.0) -- (3.1,0.0);
			\draw[->] (0.0,-0.1) -- (0.0,1.1); 
			
			\draw (-2.0,0.03) -- (-2.0,-0.03) node [below] {$-2$};
			\draw (-1.0,0.03) -- (-1.0,-0.03) node [below] {$-1$};
			\draw (1.0,0.03) -- (1.0,-0.03) node [below] {$1$};
			\draw (2.0,0.03) -- (2.0,-0.03) node [below] {$2$};
			\draw (3.0,0.03) -- (3.0,-0.03) node [below] {$3$};
			\draw (0.03,1.0) -- (-0.03,1.00) node [left] {$1$};
			
			\draw[thick] (0.0, 1.0) -- (1.0, 1.0);
			\node at (0.5,1.35) {$v_{-1, 0}^{\text{per}}$};
			\draw[->] (0.5, 1.225) -- (0.5, 1.075);
			
			\draw (-2.0, 0.0) -- (-1.0, 1.0) -- (0.0, 0.0);
			\node at (-1.0, 1.2) {$v_{-1, -1}$};
			
			\draw[blue] (-1.0, 0.0) -- (0.0, 1.0);
			\draw[blue, thick] (0.0, 1.0) -- (1.0, 0.0);
			\node at (0.0, 1.2) {\textcolor{blue}{$v_{-1, 0}$}};
			
			\draw[red] (1.0, 1.0) -- (2.0, 0.0);
			\draw[red, thick] (0.0, 0.0) -- (1.0, 1.0);
			\node at (1.0, 1.2) {\textcolor{red}{$v_{-1, 1}$}};
			
			\draw (1.0, 0.0) -- (2.0, 1.0) -- (3.0, 0.0);
			\node at (2.0, 1.2) {$v_{-1, 2}$};
		\end{tikzpicture}
		\caption{The Faber basis functions of level $j = -1$ in the periodic and nonperiodic setting.} \label{fig_Faber-1}
	\end{figure}
	
	Of course, for $\Fab({[0, 1]})$ and $\Fab(\bbT)$ we restrict the functions $v_{j, k}$ to the domain $[0, 1]$ and $\bbT$ respectively. Indeed, one easily sees that
	$$
	\bbD_j = \{k \in \bbZ: (0, 1) \cap \supp v_{j, k} \neq \es\}.
	$$
	
	\subsection{The tensor Faber basis}
	\label{sec:tfaber}
	Let now $f(x_1,\dots,x_d)$ be a $d$-variate continuous function $f$ on $[0,1]^d$. By fixing all variables except $x_i$
	we obtain by $g(\cdot) = f(x_1,\dots,x_{i-1},\cdot,x_{i+1},\dots,x_d)$ a univariate periodic
	continuous function. By applying \eqref{f5} in every such component we obtain the point-wise
	representation
	\begin{align*}\label{repr}
		f(\bx) = \sum\limits_{\bj\in \bbN_{-1}^d} \sum\limits_{\bk\in \bbD_{\bj}} d_{\bj,\bk}(f)
		v_{\bj,\bk}(\bx) \quad (\bx\in [0,1]^d),
	\end{align*}
	where 
	$$
	v_{\bj,\bk}(x_1,\dots,x_d) \coloneqq v_{j_1,k_1}(x_1) \dots v_{j_d,k_d}(x_d) \quad (\bj\in \bbN_{-1}^d, \bk\in \bbD_\bj)
	$$
	and for $f\in C(\bbR^d)$ as well as $\bj\in \bbN_{-1}^d$ and $\bk\in \bbZ^d$ we put
	\begin{equation}\label{f100}
		d_{\bj,\bk}(f) = d^{(1)}_{j_1,k_1} \cdots  d^{(d)}_{j_d,k_d}(f) 
	\end{equation}
	with $d^{(i)}_{\ell, m}:C_b(\bbR^d) \to C_b(\bbR^d)$ defined by
	$$
	d^{(i)}_{\ell,m}f(\bx) = \begin{cases}
		-\frac{1}{2}\Delta^{2,i}_{2^{-{\ell-1}}}f(x_1,...,x_{i-1},2^{-\ell}m,x_{i+1},...,x_d)& ,\ell \geq 0, m\in \mathbb{Z} \\
		f(x_1,...,x_{i-1},m,x_{i+1},...,x_d) & ,\ell=-1, m\in \mathbb{Z},
	\end{cases}
	$$
	where $\Delta^{2,i}_{2^{-\ell-1}}f$ denotes the univariate difference operator $\Delta^2_h$ with steplength $h=2^{-{(\ell+1)}}$ given by \eqref{eq:4} applied to the $i$-th component $x_i$ of the $d$-variate continuous function $f$ with all the other variables kept fix.
	
	Analogous to Definition \ref{def:univar_Faber} we can define the \emph{tensor product Faber basis over $\bbR^d$, $[0, 1]^d$, $\bbT^d$}, respectively, by
	\begin{align*}
		\Fab({\bbR^d}) & \coloneqq \{v_{\bj, \bk}: \bj \in \bbN_{-1}^d, \bk \in \bbZ^d\}, \\
		\Fab({[0, 1]^d}) & \coloneqq \{v_{\bj, \bk}: \bj \in \bbN_{-1}^d, \bk \in \bbD_\bj\}, \\
		\Fab({\bbT^d}) & \coloneqq \{v_{\bj, \bk}^{\text{per}}: \bj \in \bbN_{-1}^d, \bk \in \bbD_\bj^{{\text{per}}}\},
	\end{align*}
	where $\bbD_\bj \coloneqq \prod_{i=1}^d \bbD_{j_i}$ and analogously
	$$
	v_{\bj, \bk}^{\text{per}} (x_1, \dots, x_d) \coloneqq \prod_{i=1}^d v_{j_i, k_i}^{\text{per}}(x_i) \quad \text{for} \quad \bk \in \bbD_\bj^{\text{per}} \coloneqq \prod_{i=1}^d \bbD_{j_i}^{\text{per}}.
	$$
	
	Later on we will be interested in function spaces consisting of functions with a prescribed decay of their Faber coefficients. Our precise definitions go as follows.
	
	
	\begin{defi} \label{def:seq_space}
		Define the Banach spaces
		\begin{align*}
			\seqsp(\bbR^d) &\coloneqq \left\{ f  \in C_b(\bbR^d): \|f\|_{\seqsp(\bbR^d)} < \infty\right\},\\
			\seqsp([0, 1]^d) &\coloneqq \left\{ f \in C_b([0,1]^d):  \|f\|_{\seqsp([0,1]^d)} < \infty \right\},\\
			\seqsp(\bbT^d) &\coloneqq \left\{ f  \in C_b(\bbT^d): \|f\|_{\seqsp(\bbT^d)} < \infty \right\},
		\end{align*}
		with corresponding norms given by
		\begin{align*}
			\|f\|_{\seqsp(\bbR^d)} &\coloneqq \sup_{\bj \in \bbN_{-1}^d} 2^{ 3|\bj|_1/2} \left(\sum_{\bk \in \bbZ^d} |d_{\bj, \bk}(f)|^2\right)^{1/2},\\
			\|f\|_{\seqsp([0,1]^d)} &\coloneqq \sup_{\bj \in \bbN_{-1}^d} 2^{ 3|\bj|_1/2} \left(\sum_{\bk \in \bbD_{\bj}} |d_{\bj, \bk}(f)|^2\right)^{1/2},\\
			\|f\|_{\seqsp(\bbT^d)} &\coloneqq \sup_{\bj \in \bbN_{-1}^d} 2^{ 3|\bj|_1/2} \left(\sum_{\bk \in \bbD_{\bj}^{\text{per}}} |d_{\bj, \bk}(f)|^2\right)^{1/2}.
		\end{align*}
	\end{defi}
	
	The term ``dyad'' indicates that the functions are sampled only at points of the dyadic grid. Note the continuous, isomorphic embedding
	\begin{align} \label{eq:seqsp_embeddings}
		\seqsp([0, 1]^d) \hookrightarrow \seqsp(\bbR^d).
	\end{align}
	Also note the continuous embeddings
	\begin{align} \label{eq:continuous_embedding}
		\seqsp(\Omega) \hookrightarrow C_b(\Omega),
	\end{align}
	where $\Omega \in \{\bbR^d, [0, 1]^d, \bbT^d\}$ and $C_b(\Omega)$ is equipped with the $\sup$-norm $\|\cdot\|_\infty$. Indeed, every function $f: \Omega \rightarrow \bbR$ of the form
	$$
	f(\bx) = \sum_{\bj \in \bbN_{-1}^d} \sum_{\bk} d_{\bj, \bk}(f) v_{\bj, \bk}(\bx),
	$$
	whose coefficients decay as prescribed by $\seqsp(\Omega)$ (where the sums over $\bk$ range over the respective index sets) fulfills
	\begin{align*}
		\|f\|_\infty & \leq \sum_{\bj \in \bbN_{-1}^d} \left\|\sum_{\bk} d_{\bj, \bk}(f) v_{\bj, \bk}\right\|_\infty \lesssim \sum_{\bj \in \bbN_{-1}^d} \sup_{\bk} |d_{\bj, \bk}(f)| \\
		& \leq \sum_{\bj \in \bbN_{-1}^d} \left(\sum_{\bk} |d_{\bj, \bk}(f)|^2\right)^{1/2} \leq \sum_{\bj \in \bbN_{-1}^d} 2^{-3|\bj|_1/2} \|f\|_{\seqsp(\Omega)}
	\end{align*}
	(using the fact that for a given $\bj$ and every $\bx$ there are at most $2^d$ indices $\bk$ with $v_{\bj, \bk}(\bx) \neq 0$). Since the appearing geometric series converges to a constant we obtain the desired embedding \eqref{eq:continuous_embedding}.
	
	The following result is already implicitly contained in \cite[Lemma 3.2, Lemma 3.3, (3.26)]{BuGr04}. For the convenience of the reader we give a complete proof. 
	
	\begin{thm} \label{thm:FS_charact.}For any function $f \in \Hmix^2([0,1]^d)$ it holds 
		$$
		\|f\|_{\seqsp([0, 1]^d)} = \sup\limits_{\bj \in \bbN_{-1}^d} 2^{ 3|\bj|_1/2} \left(\sum_{\bk \in \bbD_{\bj}} |d_{\bj, \bk}(f)|^2\right)^{1/2} 
		\lesssim \|f\|_{\Hmix^2([0, 1]^d)}.
		$$
	\end{thm}
	
	This result follows from \eqref{eq:seqsp_embeddings} and its $\bbR^d$-version below using \eqref{eq:domain}. 
	
	\begin{prop}\label{prop:Faber} Let $\Hmix^2(\bbR^d)$ be normed as in \eqref{eq:10} and $f\in \Hmix^2(\bbR^d)$\,. Then  
		\begin{equation}\label{eq:8}
			\|f\|_{\seqsp(\bbR^d)} = \sup\limits_{\bj \in \bbN_{-1}^d} 2^{ 3|\bj|_1/2} \left(\sum_{\bk \in \bbZ^d} |d_{\bj, \bk}(f)|^2\right)^{1/2} 
			\lesssim \|f\|_{\Hmix^2(\bbR^d)}.
		\end{equation}
	\end{prop}
	
	\begin{rem}
		In the case of $1/2<s<2$ we have a somewhat stronger relation, see \cite{HMOU2016} and \cite{Tr10}, namely
		$$
		\left(\sum\limits_{\bj \in \bbN_{-1}^d} 2^{ 2(s-1/2)|\bj|_1} \sum_{\bk \in \bbZ^d} |d_{\bj, \bk}(f)|^2\right)^{1/2} \asymp \|f\|_{\Hmix^s(\bbR^d)}.
		$$
		Note that for fractional $s>0$ the space $\Hmix^s$ cannot be defined via weak derivatives. One would rather use Fourier analytic building blocks, see \cite{ST1987}.
	\end{rem}
	
	\begin{proof}[Proof of Proposition \ref{prop:Faber}] {\em Step 1. (Reduction to second derivatives)} By a straight-forward calculation we obtain that for a function $g \in L_2(\bbR)$ having two weak derivatives in $L_2(\bbR)$, that is $g \in H^2(\bbR)$, and for $j \in \bbN_0$ and $k\in \bbZ$ 
		\begin{equation}\label{eq:1}
			\Delta^2_{2^{-j-1}}g(2^{-j}k) = 2^{-j}\int_{\bbR}v_{j,k}(t)g^{(2)}(t)\,\diffd t.
		\end{equation}
		We first show the estimate for functions $f$ from the Schwartz class $S(\bbR^d)$, so that we may apply \eqref{eq:1} in every component of $f$. This will help us control the coefficient $|d_{\bj,\bk}(f)|$. It turns out to be rather technical in full generality. So we only consider those $\bj$, where $j_i\geq 0$ for $i=1,...,m$ and $j_i = -1$ for $i=m+1,...,d$. As it will turn out later there is no loss of generality. In fact, we obtain from the identity \eqref{eq:1} and the Cauchy--Schwarz inequality
		$$
		|\Delta^2_{2^{-j-1}}g(2^{-j}k)| \leq 2^{-3j/2}\left(\int_{I_{j,k}} \left|g^{(2)}(t)\right|^2\,dt\right)^{1/2}.
		$$
		We denote with $I_{j,k}$ the dyadic interval $[2^{-j}k,2^{-j}(k+1))$.
		This transfers to the particularly chosen $\bj\in \bbN_{-1}^d$ as follows
		\begin{align*}
			& 2^{3|\bj|_1/2}\left(\sum\limits_{\bk \in \bbZ^d}|d_{\bj,\bk}(f)|^2\right)^{1/2} \\
			\lesssim & \Bigg(\sum\limits_{\bk \in \bbZ^d}\int_{I_{j_1,k_1}}\cdots \int_{I_{j_m,k_m}}\left|f^{(2,\dots,2,0,\dots,0)}(t_1,\dots,t_m,k_{m+1},\dots,k_d)\right|^2 \\
			& ~~~~~~~~~~~~~~~~~~~~~~~~~~~~~~~~~\diffd t_1\dots \diffd t_m\Bigg)^{1/2}.
		\end{align*}
		For $m=d$ the right hand side is equal to $\|f^{(2,\dots,2)}\|_{L_2(\bbR^d)}$ and we are done. For $m<d$ we have to take care of the point evaluations.
		
		{\em Step 2. (Decomposition of the Fourier spectrum)} Now we put $g:=f^{(2,\dots,2,0,\dots,0)}$, fix $t_1\in I_{j_1,k_1},\dots,t_m \in I_{j_m,k_m}$ and decompose the function 
		$$
		g(t_1,\dots,t_m,\cdot,\dots,\cdot)= f^{(2,\dots,2,0,\dots,0)}(t_1,\dots,t_m,\cdot,\dots,\cdot),
		$$
		which is smooth and maps $\bbR^{d-m}$ to $\bbC$, into Fourier analytic building blocks 
		\begin{align}\label{eq_fourierblocks}
			g(t_1,\dots,t_m,\cdot,\dots,\cdot) = \sum\limits_{\bl \in \bbN_0^{d-m} } g^{}_{\bl}(t_1,\dots,t_m,\cdot,\dots,\cdot), 
		\end{align}
		where $g^{}_{\bl}(t_1,\dots,t_m,\cdot,\dots,\cdot) \coloneqq \mathcal{F}_{d-m}^{-1}[\chi^{}_{J_{\bl}}\mathcal{F}_{d-m}  [g(t_1,\dots,t_m,\cdot,\dots,\cdot)]]$ with 
		$$
		J_{\bl}:= J_{\ell_1}\times...\times J_{\ell_ {d-m}}
		$$
		and $J_{0} = (-1,1)$, $J_{\ell} = (-2^{\ell},-2^{\ell-1}] \cup [2^{\ell-1},2^{\ell})$ if $\ell\geq 1$. The convergence takes place in $L_2(\bbR^{d-m})$. 
		Using triangle inequality we obtain 
		\begin{align} \label{eq:3}
			\begin{split}
				&2^{3|\bj|_1/2}\left(\sum\limits_{\bk \in \bbZ^d}|d_{\bj,\bk}(f)|^2\right)^{1/2} \\
				\lesssim & \sum\limits_{\bl \in \bbN_0^{d-m}}\left(\sum\limits_{\bk \in \bbZ^d}\int_{I_{j_1,k_1}}\cdots \int_{I_{j_m,k_m}}|g_{\bl}(t_1,\dots,t_m,k_{m+1},\dots,k_d)|^2\,\diffd t_1 \dots \diffd t_m\right)^{1/2}.
			\end{split}
		\end{align}
		
		For fixed $t_1,\dots,t_m$ the function $g_{\bl}(t_1,\dots,t_m,\cdot,\dots,\cdot)$ has Fourier support in $[-2^{\ell_1},$ $2^{\ell_1}]\times \dots \times [-2^{\ell_{d-m}},2^{\ell_{d-m}}]$.
		
		{\em Step 3. (Application of Shannon's sampling theorem)} The multivariate and anisotropic version of Shannon's sampling theorem, see for instance \cite[1.5.5]{ST1987}, 
		gives
		\begin{align}\label{eq:2}
			\begin{split}
				& g^{}_{\bl}(t_1,\dots,t_m,x_1,\dots,x_{d-m}) \\
				= & \sum\limits_{k_1\in\bbZ} \dots \sum\limits_{k_{d-m}\in \bbZ} g^{}_{\bl}\left(t_1, \dots ,t_m,\frac{k_1}{2^{\ell_1}}, \dots ,\frac{k_{d-m}}{2^{\ell_{d-m}}}\right) \prod_{i=1}^{d-m} \sinc\left(2^{\ell_i}x_i - k_i\right).
			\end{split}
		\end{align}
		It is well-known that the system of $(d-m)$-variate functions 
		$$
		\left\{\prod_{i=1}^{d-m} 2^{\ell_i/2}\sinc\left(2^{\ell_i} x_i - k_i\right)\right\}_{\bk \in \bbZ^{d-m}}
		$$
		is orthonormal in $L_2(\bbR^{d-m})$. Hence we get from Parsevals's identity and \eqref{eq:2}
		\begin{align*}
			\begin{split}
				\sum\limits_{k_1\in\bbZ} \dots \sum\limits_{k_{d-m}\in \bbZ}
				2^{-\ell_1} \dots  2^{-\ell_{d-m}} & \left|g_{\bl}\left(t_1, \dots ,t_m,\frac{k_1}{2^{\ell_1}}, \dots ,\frac{k_{d-m}}{2^{\ell_{d-m}}}\right)\right|^2\\ 
				= & \|g_{\bl}(t_1,...,t_m,\cdot, \dots ,\cdot)\|^2_{L_2(\bbR^{d-m})}.
			\end{split}
		\end{align*}
		This gives in particular 
		\begin{align*}
			& \sum\limits_{k_1\in\bbZ} \dots \sum\limits_{k_{d-m}\in \bbZ}
			\left|g_{\bl}\left(t_1, \dots ,t_m,k_1, \dots ,k_{d-m}\right)\right|^2 \\
			\lesssim & 2^{|\bl|_1}\|g_{\bl}(t_1, \dots ,t_m,\cdot, \dots ,\cdot)\|^2_{L_2(\bbR^{d-m})}.
		\end{align*}
		Plugging this into \eqref{eq:3} yields
		\begin{align*}
			& 2^{3|\bj|_1/2}\left(\sum\limits_{\bk \in \bbZ^d}|d_{\bj,\bk}(f)|^2\right)^{1/2} \\
			\lesssim & \sum\limits_{\bl \in \bbN_0^{d-m}}\Bigg(\sum\limits_{\bk' \in \bbZ^{m}}\int_{I_{j_1,k'_1}}\cdots \int_{I_{j_m,k'_m}}\|g_{\bl}(t_1, \dots ,t_m,\cdot, \dots ,\cdot)\|_{L_2(\bbR^{d-m})}^2 \\
			& ~~~~~~~~~~~~~~~~~~~~~~~~~~~~~~~~~~~~~~~~~~~~ 2^{|\bl|_1} \diffd t_1 \dots \diffd t_m\Bigg)^{1/2} \\
			= & \sum\limits_{\bl \in \bbN_0^{d-m}}2^{-3|\bl|_1/2}\Bigg(\sum\limits_{\bk' \in \bbZ^{m}}\int_{I_{j_1,k'_1}}\cdots \int_{I_{j_m,k'_m}} \|g_{\bl}(t_1, \dots ,t_m,\cdot, \dots ,\cdot)\|_{L_2(\bbR^{d-m})}^2 \\
			& ~~~~~~~~~~~~~~~~~~~~~~~~~~~~~~~~~~~~~~~~~~~~~~~~~~~~~~~ 2^{4|\bl|_1} \diffd t_1 \dots \diffd t_m\Bigg)^{1/2}\,.
		\end{align*}
		
		{\em Step 4. (Final estimates for Schwartz functions)} For all $\bl \in \bbN_0^{d-m}$ we obtained 
		\begin{align*}
			& 2^{4|\bl|_1}\|g_{\bl}(t_1, \dots ,t_m,\cdot, \dots ,\cdot)\|_{L_2(\bbR^{d-m})}^2\\
			\leq & \sum\limits_{\bn \in \bbN_0^{d-m}}2^{4|\bn|_1}\|g_{\bn}(t_1, \dots ,t_m,\cdot, \dots ,\cdot)\|_{L_2(\bbR^{d-m})}^2\\
			\lesssim &  \sum\limits_{\bn \in \bbN_0^{d-m}}
			\int_{\bbR}\cdots \int_{\bbR} |\cF_{d-m}[g(t_1, \dots ,t_m,s_{m+1}, \dots ,s_d)]|^2 \chi^{}_{J_{\bn}}(\bs) \\
			& ~~~~~~~~~~~~~~~~~~~~ \prod\limits_{i=m+1}^{d}(1+|s_i|^2)^2 \, \diffd s_{m+1} \dots \diffd s_d\\
			= & \int_{\bbR}\cdots \int_{\bbR} |\cF_{d-m}[g(t_1, \dots ,t_m,s_{m+1}, \dots ,s_d)]|^2\prod\limits_{i=m+1}^{d}(1+|s_i|^2)^2 \, \diffd s_{m+1} \dots \diffd s_d,
		\end{align*}
		where we used Plancherel's identity in the second estimate and 
		$$
		2^{4|\bn|_1} \asymp \prod_{i=m+1}^{d}(1+|s_i|^2)^2
		$$
		for $\bs \in J_{\bn}$. By well-known properties of the Fourier transform the last expression is equivalent to 
		$$
		\sum\limits_{\underline{\beta} \leq \mathbf{2}} \left\|\frac{\partial^{\beta_1+ \dots +\beta_{d-m}}}{\partial x_1^{\beta_1} \dots \partial x_{d-m}^{\beta_{d-m}}}g(t_1, \dots ,t_m,\cdot, \dots ,\cdot)\right\|_{L_2(\bbR^{d-m})}^2,
		$$
		see \eqref{eq:10}, and hence
		\begin{align*}
			& 2^{3|\bj|_1/2}\left(\sum\limits_{\bk \in \bbZ^d}|d_{\bj,\bk}(f)|^2\right)^{1/2} \\
			\lesssim & \sum\limits_{\underline{\beta} \leq \mathbf{2}}\sum\limits_{\bl \in \bbN_0^{d-m}}2^{-3|\bl|_1/2}\Bigg(\sum\limits_{\bk' \in \bbZ^m}\int_{I_{j_1,k_1'}}\cdots \int_{I_{j_m,k_m'}} \\
			& \int_{\bbR}\cdots \int_{\bbR}|f^{(2, \dots ,2,\beta_1,\dots,\beta_{d-m})}(t_1, \dots ,t_m,t_{m+1}, \dots ,t_d)|^2 \diffd t_{m+1} \dots \diffd t_d \\
			& ~~~~~~~~~~~~~~~~~~~~~~~~~~~~~~~~~~~~~~~~~~~~~~~~~~~~~~~~~~~\diffd t_1 \dots \diffd t_m\Bigg)^{1/2}.
		\end{align*}
		The sum over $\bl$ is a multi-indexed geometric series and hence convergent. It remains to observe
		\begin{align*}
			& 2^{3|\bj|_1/2}\left(\sum\limits_{\bk \in \bbZ^d}|d_{\bj,\bk}(f)|^2\right)^{1/2} \\
			\lesssim & \sum\limits_{\underline\beta \leq \mathbf2}\left(\int_{\bbR}\cdots \int_{\bbR}|f^{(2, \dots ,2,\beta_1,\dots,\beta_{d-m})}(t_1, \dots ,t_d)|^2 \diffd t_{1} \dots \diffd t_d\right)^{1/2} \\
			\lesssim & \|f\|_{\Hmix^2(\bbR^d)}.
		\end{align*}
		
		{\em Step 5. (Lift to $\Hmix^2(\bbR^d)$ via density)} By a proper reordering of the variables (and integrals) we may prove the last inequality for any $\bj \in \bbN^d_{-1}$\,. Hence, we have shown the desired inequality in \eqref{eq:8} for smooth functions $f$ from the Schwartz space $S(\bbR^d)$. What remains follows from a density argument since the Schwartz functions are dense in $\Hmix^2(\bbR^d)$. In fact, let $f \in \Hmix^2(\bbR^d)$ and $(f_n)_{n \in \bbN}\subset S(\bbR^d)$ converge to $f$ in $\Hmix^2(\bbR^d)$. By \eqref{eq:8}, the sequence $(f_n)_{n \in \bbN}$ is Cauchy in $\seqsp(\bbR^d)$ and converges to $\tilde{f} \in \seqsp(\bbR^d)$. Both spaces, $\Hmix^2(\bbR^d)$ and $\seqsp(\bbR^d)$, are continuously embedded into $C_b(\bbR^d)$ (using classical embedding results for Sobolev spaces and \eqref{eq:continuous_embedding} for the dyadic space), so that $f$ and $\tilde{f}$ have to be continuous and 
		$$
		\|f-\tilde{f}\|_{\infty} \lesssim \|f-f_n\|_{\Hmix^2(\bbR^d)} +
		\|\tilde{f}-f_n\|_{\seqsp(\bbR^d)} \to 0 
		$$
		as $n$ goes to infinity. Thus $f$ and $\tilde{f}$ coincide pointwise. Finally, we have for any $n\in \bbN$ the relations 
		\begin{equation}\nonumber
			\|f\|_{\seqsp(\bbR^d)} \lesssim \|f-f_n\|_{\seqsp(\bbR^d)}+\|f_n\|_{\Hmix^2(\bbR^d)},
		\end{equation}
		where the right-hand side converges to $\|f\|_{\Hmix^2(\bbR^d)}$ as $n$ goes to infinity. This finally shows 
		\begin{equation}\nonumber
			\|f\|_{\seqsp(\bbR^d)} \lesssim \|f\|_{\Hmix^2(\bbR^d)}
		\end{equation}
		for any $f \in \Hmix^2(\bbR^d)$.
	\end{proof}
	
	\begin{rem} With an analogous proof but more involved techniques from Fourier analysis we may also prove for $1<p<\infty$
		$$
		\sup\limits_{\bj \in \bbN_{-1}^d} 2^{ |\bj|_1(2-1/p)} \left(\sum_{\bk \in \bbZ^d} |d_{\bj, \bk}(f)|^p\right)^{1/p} \lesssim \|f\|_{S^2_pW(\bbR^d)},
		$$
		where $S^2_pW(\bbR)$ denotes the Sobolev space with dominating mixed smoothness of order two and integrability $p$, see \cite[Chapter 2]{ST1987} and \cite{Tr10}.
		
	\end{rem}

	\section{The tent transform} \label{sec:tent}
	
	The \emph{tent transform} of a function $f: [0, 1]^d \rightarrow \bbR$ is defined as
	\begin{align*}
		\cR f(y_1, \dots, y_d) = f(|2y_1-1|, \dots, |2y_d-1|).
	\end{align*}
	We can think of $\cR f$ as being defined over $\bbT^d$. Assume that $f: [0, 1]^d \rightarrow \bbR$ is continuous. Notice that sampling a transformed function $\cR f(y_1, \dots, y_d)$ is equivalent to sampling $f(|2y_1-1|, \dots, |2y_d-1|)$. In particular we have
	\begin{align} \label{eq:wce_transfer}
		\frac1N \sum_{\bx \in X} f(\bx) - \int_{[0, 1]^d} f(\bx) \, \diffd\bx = \frac1N \sum_{\by \in Y} \cR f(\by) - \int_{\bbT^d} \cR f(\by) \, \diffd\by,
	\end{align}
	where, for a given set $Y \sbse \bbT^d$, we construct the pullback under the tent transform (which might be a multiset)
	$$
	X = \cR Y \coloneqq \{(|2y_1-1|, \dots, |2y_d-1|): (y_1, \dots, y_d) \in Y\}.
	$$
	Throughout the paper we will denote points set in $\bbT^d$ by $Y$ and point sets in $[0, 1]^d$ by $X$.
	
	We now study the tent transform $\cR$ as an operator between function spaces. To this end we first present the following result.
	
	\begin{thm} \label{thm:tent}
		The tent transform embeds continuously
		$$
		\cR: \Hmix^2([0, 1]^d) \rightarrow \seqsp(\bbT^d).
		$$
	\end{thm}
	
	Combining this embedding with \eqref{eq:wce_transfer}, we see that the worst case error of the quasi-Monte Carlo rule over $\seqsp(\bbT^d)$ translates to upper bounds for the worst case error over $\Hmix^2([0, 1]^d)$ under the tent transformation.
	
	\begin{proof}[Proof of Theorem \ref{thm:tent}]
		Note that a tent transformed element of the nonperiodic Faber basis $v_{\bj, \bk} \in \Fab({[0, 1]^d})$ can be expressed as a finite, linear combination of elements of the periodic Faber basis $v_{\bj, \bk}^{\text{per}} \in \Fab({\bbT^d})$ (one level higher). To be precise, in the one dimensional case we have for $j \ge 0, k \in \bbD_j$
		\begin{align*}
			\cR v_{-1, 0} & = v_{0, 0}^{\text{per}}, \\
			\cR v_{-1, 1} & = v_{-1, 0}^{\text{per}} - v_{0, 0}^{\text{per}}, \\
			\cR v_{j, k} & = v_{j+1, 2^j + k}^{\text{per}} + v_{j+1, 2^j-k-1}^{\text{per}}.
		\end{align*}
		A comparison of coefficients in the Faber representation shows that the coefficients transform under $\cR$ by (with $j \ge 1$)
		\begin{align}\label{tentdiff}
			\begin{split}
				d_{-1,0}(\cR f)  &=d_{-1,1}(f),\\
				d_{0,0} (\cR f) &= d_{-1,0}(f)-d_{-1,1}(f),\\
				d_{j,k} (\cR f) &= 
				\begin{cases}
					d_{j-1,2^{j-1}-k-1}(f) &, 0 \le k < 2^{j-1}  \\
					d_{j-1,k-2^{j-1}}(f) &, 2^{j-1} \le k < 2^j.
				\end{cases}
			\end{split}
		\end{align}
		We see that for each $j\ge 1$ and $k\in \{2^{j-1}, 2^{j-1}+1, \dots, 2^j-1\}$ there is exactly one $\tilde k\in \bbD_{j-1}$ such that
		\begin{align}\label{sumtentdiff}
			d_{j,k}(\cR f)= d_{j,\tilde k}(\cR f).
		\end{align}
		In the multivariate case, we use the product structure of the mixed differences \eqref{f100} and use the results for univariate functions. Fix $\bj \in \bbN_{-1}^d$. Define $\phi(\bj) \coloneqq ((j_i)_+ - 1)_{i=1}^d$ and decompose $[d]=M_-\cup M_0 \cup M_+$ into disjoint (possibly empty) sets
		\begin{align*}
			M_{-}\coloneqq\{i:j_i=-1\}, \quad M_0\coloneqq\{ i:j_i=0 \}, \quad M_+\coloneqq\{i:j_i\ge1\}.
		\end{align*}
		By sorting the product according to the above decomposition and applying \eqref{sumtentdiff} respectively \eqref{tentdiff} we get
		\begin{align*}
			\sum_{\bk\in \bbD_\bj^{\text{per}}} &|d_{\bj,\bk}(\cR f)|^2 
			= \sum_{\bk\in \bbD_\bj^{\text{per}}} \left|  \prod_{i \in M_-}d^{(i)}_{-1,0}  \prod_{i \in M_0}d^{(i)}_{0,0}  \prod_{i \in M_+}d^{(i)}_{j_i,k_i} (\cR f)  \right|^2\\
			&\lesssim \sum_{\bk\in \bbD_{\phi(\bj)}^{\text{per}}}\left| \prod_{i \in M_-}d^{(i)}_{-1,1}  \prod_{i \in M_0}(d^{(i)}_{-1,0}-d^{(i)}_{-1,1}) \prod_{i \in M_+}d^{(i)}_{j_i-1,k_i}  (f)  \right|^2 \\
			&\lesssim \sum_{\bk\in \bbD_{\phi(\bj)}^{\text{per}}}\left| \prod_{i \in M_-}d^{(i)}_{-1,1}  \sum_{A\subseteq M_0}(-1)^{\# A} \prod_{i \in M_0}d^{(i)}_{-1,\chi^{}_A(i)}   \prod_{i \in M_+}d^{(i)}_{j_i-1,k_i} (f)  \right|^2 \\       
			&\lesssim \sum_{\bk \in \bbD_{\phi(\bj)}}  \left| d_{\phi(\bj),\bk}(f) \right|^2,
		\end{align*}
		and obtain the desired
		$$
		\sup_{\bj \in \bbN_{-1}^d} 2^{3|\bj|_1/2}\left( \sum_{\bk \in \bbD_{\bj}^{\text{per}}} |d_{\bj,\bk}(\cR f)|^2 \right)^{1/2} 
		\lesssim \sup_{\bj \in \bbN_{-1}^d} 2^{3|\bj|_1/2}\left( \sum_{\bk \in \bbD_{\bj}} |d_{\bj,\bk}(f)|^2 \right)^{1/2}.
		$$
		Thus $\cR$ defines a continuous operator between $\Hmix^2([0, 1]^d)$ and $\seqsp(\bbT^d)$.
	\end{proof}
	
	It is natural to ask for similar mapping properties of $\cR$ between further spaces. Such questions have been studied in the recent preprint \cite{SU2025}. Corollary 4.4 there is in a similar spirit to what we looked at here but does not quite reach smoothness $s=2$.
	
	\section{Tent transformed order 2 nets} \label{sec:nets}
	
	\subsection{General definitions of higher order digital nets}
	
	Digital nets go back to the construction of Halton \cite{Hal1964} via digit reversals. It was only realized later how this construction should be generalized in a usable way, for historical remarks see \cite{DP2010}. On the one hand, this class of point sets should be large enough to capture the important properties a  point set needs to posses for good approximation properties. On the other hand, the class should still be sufficiently structured to get a grasp of the involved analysis for the error of the corresponding quasi-Monte Carlo rule. We aim to give an accessible introduction to these point sets, further information and generalizations can be found in \cite{BD2009, Dic2008, DP2010}.
	
	Generally speaking, a digital $(t, n, d)$-net in base $b$ is a subset $Y \sbse \bbT^d$ of size $\#Y = b^n$ such that for each elementary, $d$-dimensional, $b$-adic interval
	$$
	J = \prod_{i=1}^d \left[\frac{a_i}{b^{\ell_i}}, \frac{a_i+1}{b^{\ell_i}}\right) \quad (a_i \in \{0, 1, \dots, b^{\ell_i}-1\})
	$$
	of volume $|J| = b^{-\ell_1 - \dots - \ell_d} \stackrel{!}{=} b^{t-n}$ it holds $\#(Y \cap J) = b^t$. That is every elementary interval up to a given volume contains as many points as would be expected from a ``uniformly distributed'' set. The parameter $t$ can be seen as measuring the quality of the point set, where we would like to have $t$ small. It is common to only consider prime powers $b = p^m$ as bases. In what follows we will only be interested in the case $b = 2$ (which in general also yields lower implicit constants in the asymptotic bounds).
	
	Analogous to how lattices can be generated from a generating vector $\bv \in [0, 1)^d$ by pushing it through the torus, digital nets can be constructed from generator matrices
	$$
	C_i = \left(\begin{matrix}
		- & c_{i, 1}^\top & - \\
		& \vdots & \\
		- & c_{i, n}^\top & -
	\end{matrix}\right) \in \{0, 1\}^{n \times n}
	$$
	for $i=1, \dots, d$, where $\{0, 1\} \cong \bbZ/2\bbZ$ will be interpreted as the field with $2$ elements. The points $\by^k$ in the set $Y = Y(C_1, \dots, C_n) = \{\by^0, \by^1, \dots, \by^{N-1}\} \sbse [0, 1)^d$ with $N = 2^n$ are then constructed as follows: collect the binary digits $k = \kappa_0 + 2 \kappa_1 + \dots + 2^{n-1} \kappa_{n-1}$ with $\kappa_j \in \{0, 1\}$ in a vector
	$$
	\kappa = \left(\begin{matrix}
		\kappa_0 \\
		\kappa_1 \\
		\vdots \\
		\kappa_{n-1}
	\end{matrix}\right)
	$$
	and determine, using arithmetic in $\bbZ/2\bbZ$, the product
	$$
	C_i \kappa = \left(\begin{matrix}
		c_{i, 1}^\top \kappa \\
		\vdots \\
		c_{i, n}^\top \kappa
	\end{matrix}\right) = \left(\begin{matrix}
		\kappa^{(i)}_1 \\
		\vdots \\
		\kappa^{(i)}_n
	\end{matrix}\right).
	$$
	The $i$-th component of $\by^k$ is then given by
	$$
	y_i^k = \frac{\kappa_1^{(i)}}{2^1} + \frac{\kappa_2^{(i)}}{2^2} + \dots + \frac{\kappa_n^{(i)}}{2^n}.
	$$
	Given a set of generator matrices $C_i \in \{0, 1\}^{n \times n}, i=1, \dots, d$ the parameters $n$ and $d$ can be seen immediately. The quality parameter $t$ of $Y$ translates into the following property for the generator matrices \cite{DP2010}: for any $0 \leq \ell_1, \dots, \ell_d \leq n$ with $\ell_1 + \dots + \ell_d = n-t$ the initial rows
	\begin{align} \label{linear_independence_parameter}
		c^{}_{1, 1}, \dots, c^{}_{1, \ell_1}, \dots, c^{}_{d, 1}, \dots, c^{}_{d, \ell_d}
	\end{align}
	of the generator matrices are linearly independent. Clearly, the smaller the value of $t$ is, the stronger this property gets. Indeed, the smallest possible $t$ for which \eqref{linear_independence_parameter} holds is also the smallest $t$ for which $Y$ is a $(t, n, d)$-net.
	
	To improve the quality of such nets even further one introduces the order parameter $\alpha$ \cite{BD2009, Dic2008, HMOU2016}. We will formulate it in terms as we will need it later on (also see \cite{GSY16, GSY2016}).
	
	\begin{defi}
		Let $t, n, d, \alpha \in \bbN$ and let $C_1, \dots, C_d \in \{0, 1\}^{\alpha n \times n}$ be matrices with rows $c_{i, j}^\top, i=1, \dots, d, j=1, \dots, \alpha n$ respectively. Let $Y \sbse \bbT^d, \#Y = 2^n$ be the digital net generated by the matrices $C_1, \dots, C_d$ as above (with the obvious modification for non-square matrices).
		
		The point set $Y$ is called an \emph{order $\alpha$ digital $(t, n, d)$-net} if for all indices $1 \leq j_{i, \nu_i} < \dots < j_{i, 1} \leq \alpha n$, where $0 \leq \nu_i \leq n$, with
		$$
		\sum_{i=1}^d \sum_{\ell = 1}^{\min\{\alpha, \nu_i\}} j_{i, \ell} \leq \alpha n - t
		$$
		the row vectors
		$$
		c_{1, j_{1, \nu_1}}^\top, \dots, c_{1, j_{1, 1}}^\top, \dots, c_{d, j_{d, \nu_d}}^\top, \dots, c_{d, j_{d, 1}}^\top \in \{0, 1\}^n
		$$
		are linearly independent over $\{0, 1\} \cong \bbZ/2\bbZ$.
	\end{defi}
	
	Note how the indices $j_{i, 1} > \dots > j_{i, \nu_i}$ are ordered descendingly, so that we only constrain how far down the (at most) $\alpha$ last rows can be that we pick out of every matrix $C_i$. If $\nu_i \geq \alpha$ we thus may assume that also $j_{i, \alpha+1} = j_{i, \alpha}-1$, $j_{i, \alpha+2}=j_{i, \alpha}-2$, $\dots$, $j_{i, \nu_i} = 1$. If $\alpha = 1$ we recover the definition of a digital $(t, n, d)$-net.
	
	Again, the parameter $t$ controls the quality of the point set where we would like to have $t$ small. The parameter $\alpha$ is connected to the smoothness of the underlying functions that we want to integrate via the quasi-Monte Carlo rule \cite{Dic2008, GSY2016}. Also, $\alpha$ can be interpreted as a \emph{complexity parameter} since every component of a point of an order $\alpha$ digital $(t, n, d)$-net requires $\alpha \log_2 N$ bits to write down in binary.
	
	
	
	\subsection{Constructions} \label{subsec:construction}
	
	Higher order nets can be constructed from standard digital nets via the following \cite{BD2009, Dic2008, HMOU2016}. Let $\tilde Y = \{\tilde\by^0, \tilde\by^1, \dots, \tilde\by^{N-1}\} \in \bbT^{\alpha d}$, $N = 2^{n}$ be a $(\tilde t, n, \alpha d)$-net (generated from square matrices $C_i \in \{0, 1\}^{n \times n}, i=1, \dots, \alpha d$). Every component of $\tilde \by^k$ can be written in base $2$ using only $n$ digits after the comma
	$$
	\tilde y^k_i = \sum_{\ell = 1}^{n} \frac{\eta^k_{i, \ell}}{2^\ell} \quad (\eta^k_{i, \ell} \in \{0, 1\}).
	$$
	Define now $Y = \{\by^0, \by^1, \dots, \by^{N-1}\} \sbse \bbT^d$ with $N = 2^n$ as above via
	$$
	y_i^k = \sum_{\ell = 1}^{n} \sum_{s = 1}^\alpha \frac{\eta^{k}_{(i-1)\alpha + s, \ell}}{2^{(\ell-1)\alpha+s}}
	$$
	for $i=1, \dots, d$ and $k=0, 1, \dots, N-1$. This operation interlaces the binary digits of $\tilde y^k_{(i-1)\alpha + 1}, \dots, \tilde y^k_{i\alpha}$ giving a number with $\alpha n$ digits after the comma. The resulting set $Y$ is then an order $\alpha$ digital $(t, n, d)$-net where $t$ can be chosen as \cite{BD2009, Dic2008, HMOU2016}
	$$
	t = \alpha \tilde t + \frac{d \alpha (\alpha-1)}2.
	$$
	Starting with a digital net we can thus construct digital nets of arbitrarily high order. Although note that this also raises the quality parameter $t$.
	
	It is a nontrivial task to construct digital nets with a good quality parameter, especially as the dimension $d$ increases. There are multiple constructions available by now. Digital sequences $(\by^0, \by^1, \dots)$ are, roughly speaking, sequences of points such that any initial block of $2^n$ elements constitutes a digital net, where the corresponding quality parameter $t$ could depend on $n$, ideally in a controllable manor. For precise definitions consult \cite{DP2010}. Examples of such sequences include the Sobol'-construction \cite{Sob1967} and the Niederreiter-construction \cite{Nie1988}. The ones that we will use and that generally yield the lowest $t$ values are the Niederreiter--Xing sequences \cite{NX1996, XN1995}. These sequences use sophisticated, algebraic tools to construct the generator matrices. Fortunately, computer implementations are available \cite{Pir2002}. We will use the database \cite{KN2016} available under
	\begin{center}
		\textit{people.cs.kuleuven.be/$\sim$dirk.nuyens/qmc-generators/}
	\end{center}
	which contains implementations for these sequences in Matlab/Octave, C++ and Python2/3. Given a digital sequence $\{\by^0, \by^1, \dots \}$ in dimension $\alpha d-1$ from the Niederreiter-Xing construction we can construct a $(t, n, \alpha d)$-net via
	$$
	\{(\lfloor y^k_1\rfloor_n, \dots, \lfloor y^k_{\alpha d-1}\rfloor_n, k/2^n): k=0, 1, \dots, 2^n-1\},
	$$
	where $\lfloor y \rfloor_n$ truncates all but the first $n$ binary digits of $y$ after the comma. Using the digital interlacing from above we can turn this into an order $\alpha$ digital $(t, n, d)$-net (with a possibly different $t$ value), in particular for $\alpha = 2$.
	
	\subsection{The worst case error for tent transformed order 2 nets}
	
	We now come to our central result on tent transformed order $2$ nets, giving a bound for the decay for the quasi-Monte Carlo rule. 
	
	\begin{thm} \label{thm:upper_bound}
		Let $Y \sbse \bbT^d$ be an order $2$ digital $(t, n, d)$-net of cardinality $\#Y = N = 2^n$. Let $X = \cR Y \sbse [0, 1]^d$ be the tent transformed point set with $\#X = N$. Then
		$$
		\QMC(X, \Hmix^2([0, 1]^d)) \lesssim 2^t N^{-2}(\log N)^{2d-1},
		$$
		where the implicit constant depends only on $d$.
	\end{thm}
	
	\begin{proof}
		We will follow a similar strategy as in \cite{HMOU2016} with some adaptations for the case of smoothness $2$, using Proposition \ref{prop:Faber}.
		
		Let $f \in \Hmix^2([0, 1]^d)$, we need to upper bound the quantity
		$$
		\left|\frac1N \sum_{\bx \in X} f(\bx) - \int_{[0, 1]^d} f(\bx) \, \diffd \bx\right| = \left|\frac1N \sum_{\by \in Y} \cR f(\by) - \int_{\bbT^d} \cR f(\by) \, \diffd \by\right|.
		$$
		Let $g = \cR f \in \seqsp(\bbT^d)$. Writing
		$$
		g = \sum_{\bj \in \bbN_{-1}^d} \sum_{\bk \in \bbD_\bj^{\text{per}}} d_{\bj, \bk}(g) v_{\bj, \bk}^\text{per}
		$$
		we get
		\begin{align*}
			& \frac1N \sum_{\by \in Y} g(\by) - \int_{\bbT^d} g(\by) \, \diffd\by \\
			= & \sum_{\bj \in \bbN_{-1}^d} \sum_{\bk \in \bbD_\bj^{\text{per}}} d_{\bj, \bk}(g) \left(\frac1N \sum_{\by \in Y} v_{\bj, \bk}^\text{per}(\by) - \int_{\bbT^d} v_{\bj, \bk}^\text{per}(\by) \, \diffd \by\right).
		\end{align*}
		Setting
		$$
		c_{\bj, \bk}(Y) \coloneqq \frac1N \sum_{\by \in Y} v_{\bj, \bk}^\text{per}(\by) - \int_{\bbT^d} v_{\bj, \bk}^\text{per}(\by) \, \diffd \by,
		$$
		we obtain
		\begin{align*}
			& \left|\frac1N \sum_{\bx \in X} f(\bx) - \int_{[0, 1]^d} f(\bx) \, \diffd \bx\right| \leq \sum_{\bj \in \bbN_{-1}^d} \sum_{\bk \in \bbD_\bj^{\text{per}}} |d_{\bj, \bk}(g) c_{\bj, \bk}(Y)| \\
			\leq & \sum_{\bj \in \bbN_{-1}^d} \left(\sum_{\bk \in \bbD_\bj^{\text{per}}} |d_{\bj, \bk}(g)|^2\right)^{1/2} \left(\sum_{\bk \in \bbD_\bj^{\text{per}}} |c_{\bj, \bk}(Y)|^2\right)^{1/2} \\
			\leq & \|g\|_{\seqsp(\bbT^d)} \sum_{\bj \in \bbN_{-1}^d} 2^{-3|\bj|_1/2} \left(\sum_{\bk \in \bbD_\bj^{\text{per}}} |c_{\bj, \bk}(Y)|^2\right)^{1/2}.
		\end{align*}
		By Theorem \ref{thm:tent} we have $\|g\|_{\seqsp(\bbT^d)} \lesssim \|f\|_{\Hmix^2([0, 1]^d)}$, so that it remains to show
		$$
		\sum_{\bj \in \bbN_{-1}^d} 2^{-3|\bj|_1/2} \left(\sum_{\bk \in \bbD_\bj^{\text{per}}} |c_{\bj, \bk}(Y)|^2\right)^{1/2} \lesssim \frac{(\log N)^{2d-1}}{N^2}.
		$$
		Decompose the sum
		$$
		\sum_{\bj \in \bbN_{-1}^d} = \sum_{\substack{\bj \in \bbN_{0}^d \\ |\bj|_1 < M}} + \sum_{\substack{\bj \in \bbN_{0}^d \\ |\bj|_1 \geq M}} + \sum_{e \sbsn [d]} \sum_{\bj \in \bbN_{-1}^d(e)},
		$$
		where $M \coloneqq n - \lceil t/2 \rceil \leq n$ and $\bbN_{-1}^d(e) \coloneqq \{\bj \in \bbN_{-1}^d: \{i \in [d]: j_i \neq -1\} = e\}$. We will apply \cite[(5.9-12)]{HMOU2016} (in their notation $r = 2, p = p' = 2, q' = 1$), which uses a bound for $c_{\bj, \bk}(Y)$ from \cite{Mar2015}, see \cite[Proposition 4.3, Lemma 5.1, Lemma 5.2]{HMOU2016}, only valid because we work with order $2$ nets. To be specific, from (5.9) and the subsequent estimate we get
		\begin{align} \label{ineq:estimate_less_M}
			\begin{split}
				\sum_{\substack{\bj \in \bbN_{0}^d \\ |\bj|_1 < M}} 2^{-3|\bj|_1/2} \left(\sum_{\bk \in \bbD_\bj^{\text{per}}} |c_{\bj, \bk}(Y)|^2\right)^{1/2} & \lesssim 2^{-2n+t} \sum_{\ell = 0}^{M-1} \ell^{d-1} (M-\ell)^{d-1} \\
				& \lesssim 2^{-2n+t} M^{2d-1} \lesssim 2^t N^{-2} (\log N)^{2d-1}.
			\end{split}
		\end{align}
		From (5.10) and the subsequent estimate we get
		\begin{align} \label{ineq:estimate_geq_M}
			\begin{split}
				\sum_{\substack{\bj \in \bbN_{0}^d \\ |\bj|_1 \geq M}} 2^{-3|\bj|_1/2} \left(\sum_{\bk \in \bbD_\bj^{\text{per}}} |c_{\bj, \bk}(Y)|^2\right)^{1/2} & \lesssim 2^{-2n+t} n^{d-1} + 2^{-2n} n^{d-1} \\
				& \lesssim 2^t N^{-2} (\log N)^{d-1}.
			\end{split}
		\end{align}
		Finally, from (5.12) we get
		\begin{align*}
			\begin{split}
				\sum_{\bj \in \bbN_{-1}^d(e)} 2^{-3|\bj|_1/2} \left(\sum_{\bk \in \bbD_\bj^{\text{per}}} |c_{\bj, \bk}(Y)|^2\right)^{1/2} \lesssim 2^{-2n} n^{\# e - 1} \lesssim N^{-2} (\log N)^{\#e-1},
			\end{split}
		\end{align*}
		so that
		\begin{align} \label{ineq:estimate_e}
			\sum_{e \sbsn [d]} \sum_{\bj \in \bbN_{-1}^d(e)} 2^{-3|\bj|_1/2} \left(\sum_{\bk \in \bbD_\bj^{\text{per}}} |c_{\bj, \bk}(Y)|^2\right)^{1/2} \lesssim N^{-2} (\log N)^{d-2}.
		\end{align}
		Putting \eqref{ineq:estimate_less_M}, \eqref{ineq:estimate_geq_M} and \eqref{ineq:estimate_e} together yields the claim.
	\end{proof}
	
	\section{The tent transform and classical Korobov spaces} \label{sec:Korobov}
	
	In this section we want to give another mapping property of the tent transform for spaces relevant to quasi-Monte Carlo quadrature and only requiring Fourier information. We define the \emph{classical Korobov space} $\bE^s_d \coloneqq \bE^s(\bbT^d)$ with smoothness $s>0$ via the norm
	$$
	\|f\|_{\bE^s_d} \coloneqq \sup\limits_{\bk \in \bbZ^d}|\hat{f}(\bk)| 
	\prod\limits_{i=1}^d (1+|k_i|)^s,
	$$
	where
	$$
	\hat f(\bk) = \int_{[0, 1]^d} f(\bx) \exp(-2\pi\iu \bk^\top \bx) \,\diffd \bx
	$$
	denotes the Fourier coefficient over $\bbT^d$ (contrary to $\cF[\cdot]$ which denotes the Fourier transform over $\bbR$). One sometimes finds a different notion of ``Korobov space'' in the literature, see for example \cite{Dic2008} or \cite{DSWW2006} (which in our notation corresponds to $\Hmix^s(\bbT^d)$ or $\Hmix^s([0, 1]^d)$). Here we use the space as defined in \cite{Kor1959} and \cite[Section 3.3]{DTU18} and call it ``classical Korobov space'' to emphasize this distinction. In the sequel we will prove that the tent transform embeds from $\Hmix^2([0,1]^d)$ into $\bE^2_d$. As preparation for the general setting let us start with the case for $d=1$. 
	
	\begin{lem}\label{lem_RmapsH2mixtoE21} In the univariate case the tent transform $\mathcal{R}$ maps $\Hmix^2([0,1])$ continuously to $\bE^2_1$, that is
		$$
		\mathcal{R}:\Hmix^2([0,1]) \to \bE^2(\bbT). 
		$$    
	\end{lem}
	\begin{proof} 
		It is clear that
		$$
		(\cR f)^\wedge(0) = \int_{0}^1 f(|2x-1|) \,\diffd x = \int_{0}^1 f(x) \, \diffd x,
		$$
		so that $|(\cR f)^\wedge(0)| \lesssim \|f\|_{\Hmix^2([0, 1])}$. Similarly, for $k \in \bbZ\setminus \{0\}$ we have
		\begin{align*}
			\begin{split}
				{(\mathcal{R}f)}^\wedge(k) &= \int_0^1 f(|2x-1|)\exp(-2\pi \iu kx)\,\diffd x \\
				&=\frac{1}{2} \int_{-1}^1 f(|y|)\exp(-\pi \iu k(y+1))\,\diffd y \\
				&= \frac{\exp(-\pi\iu k)}{2}\int_{-1}^1 f(|y|)\exp(-\pi \iu ky)\,\diffd y\\
				&= (-1)^k\int_{0}^1 f(t)\cos(\pi k t)\,\diffd t - 
				\frac{(-1)^k}{2} \iu \int_{-1}^1 f(|y|)\sin(\pi k y)\,\diffd y.
			\end{split}
		\end{align*}
		The second integral disappears by symmetry, so that we continue with the first integral and observe via integration by parts
		\begin{align*}
			\begin{split}
				\int_{0}^1 f(t)\cos(\pi k t)\,\diffd t &= \left[\frac{f(y)\sin(k\pi y)}{k\pi}\right]_0^1-\frac{1}{\pi k }\int_{0}^1 f'(t)\sin(\pi k t)\, \diffd t \\
				&= -\frac{1}{\pi k }\int_{0}^1 f'(t)\sin(\pi k t)\,\diffd t.
			\end{split}
		\end{align*}
		We apply partial integration again and get
		\begin{align*}
			\int_{0}^1 f'(t)\sin(\pi k t)\,\diffd t = -\frac{1}{k\pi}\left[f'(t)\cos(\pi k t)\right]_{t=0}^1 + \frac{1}{\pi k}\int_{0}^1 f''(t)\cos(\pi k t)\,\diffd t.
		\end{align*}
		Hence, we get 
		\begin{align*}
			\begin{split}
				|{(\mathcal{R}f)}^\wedge(k)| &\leq \frac{1}{|\pi k|^2}\left(|f'(1)|+|f'(0)|+\frac{1}{\sqrt 2}\|f''\|_{L_2([0,1])}\right)\,\\
				&\lesssim |k|^{-2}\|f\|_{\Hmix^2([0,1])}
			\end{split}
		\end{align*}
		and we observe the desired decay rate of $|k|^{-2}$. 
	\end{proof}
	As for the multivariate situation we will prove the following stronger result which looks similar to the one in Theorem \ref{thm:tent} (and Proposition \ref{prop:Faber}). For this we will introduce the function space $\bE^2(\bbR^d)$ over $\bbR^d$ whose defining norm will be given in the statement.
	
	\begin{thm} \label{Thm:E^2_d}
		For any function $f\in \Hmix^2(\bbR^d)$ we have 
		\begin{align*}
			\|f\|_{\bE^2(\bbR^d)} \coloneqq \sup\limits_{\bk \in \bbZ^d} \prod\limits_{i=1}^d (1+|k_i|)^2\left(\sum\limits_{\bm \in\bbZ^d}|{[\mathcal{R}f(\cdot - \bm)]}^\wedge(\bk)|^2\right)^{1/2} \lesssim \|f\|_{\Hmix^2(\bbR^d)}.
		\end{align*}
	\end{thm}
	
	\begin{proof} The proof is, at parts, similar to the one of Proposition \ref{prop:Faber}.
		
		{\em Step 1.} Analogous to the proof of Lemma \ref{lem_RmapsH2mixtoE21} for $f\in S(\bbR)$, $m \in \bbZ$ it holds
		\begin{align*}
			& [\cR f(\cdot-m)]^{\wedge}(k)\\
			= & -\frac{1}{(\pi k)^2} \left( f'(-m+1)+(-1)^{k+1}f'(-m)+\int_{-m}^{-m+1}f''(x)\cos(\pi k x) \, \diffd x \right)\\
			= & -\frac{1}{(\pi k)^2} \left( 2\chi^{}_{2\bbZ+1}(k)f'(-m)+\int_{-m}^{-m+1}f''(x)(1+\cos(\pi k x)) \, \diffd x \right)
		\end{align*}
		for $k\neq0$ and $[\cR f(\cdot-m)]^{\wedge}(0)=\int_{-m}^{-m+1}f(x) \, \diffd x$. 
		
		{\em Step 2.} For multivariate $f\in S(\bbR^d)$ by tensorization and, if necessary, by swapping differentiation and integration, reordering and using Cauchy--Schwarz inequality we obtain terms of the form
		\begin{align*}
			\prod_{i=1}^{d}\frac{1}{1+|k_i|^2}\left( \sum_{\bm \in \bbZ^n} \int_{\bbR^{d-n}} |f^{(\underline \beta)}(\bm,\bx)|^2 \, \diffd \bx \right)^{1/2}
		\end{align*}
		with $n \in \{0,\dots,d\}$ and $ \bbeta \in \{1\}^n\times\{0,2\}^{d-n}$ (In the cases were $n=0$ or $n=d$ the sum respectively the integral disappears).  To get rid of the point evaluations we fix $\bx\in \bbR^{d-n}$ and decompose $g_\bx:\bbR^n \to \bbR, \bs \mapsto f^{(\bbeta)}(\bs,\bx)$ as in \eqref{eq_fourierblocks}
		\begin{align*}
			g_\bx(s)=\sum_{\bl \in \bbN_0^{n}} g_{\bx,\bl}(s).
		\end{align*}
		and estimate 
		\begin{align*}
			\sum_{\bm \in \bbZ_{}^n}|g_{\bx,\bl}(\bm)|^2
			&\le 2^{|\bl|_1}\|g_{\bx,\bl}\|^2_{L_2(\bbR^n)}
			\le 2^{-|\bl|_1}\sum_{\bk \in \bbN_0^n} 2^{2|\bk|_1}\|g_{\bx,\bk}\|^2_{L_2(\bbR^n)}\\
			&\lesssim 2^{-|\bl|_1}\int_{\bbR^n} \left|\mathcal F_n[g_\bx](\bs)\prod_{j=1}^n (1+|s_j|)\right|^2 \,\diffd \bs \\
			&\le 2^{-|\bl|_1}\sum_{\bgamma\in \{0,1\}_{}^n}\|g_\bx^{( \bgamma)}\|^2_{L_2(\bbR^n)}.
		\end{align*}
		by using Shannon's sampling theorem and Plancherel’s identity.
		With that in mind we get
		\begin{align*}
			& \left( \sum_{\bm\in \mathbb Z^n} \int_{\bbR^{d-n}}|f^{(\bbeta)}(\bm,\bx)|^2 \, \diffd \bx\right)^{1/2}\\
			\lesssim & \sum_{\bl \in \bbN_0^n} 2^{-|\bl|_1/2} \left(\sum_{ \bgamma\in \{0,1\}^n}\int_{\bbR^{d-n}} \|g_\bx^{(\bgamma)}\|^2_{L_2(\bbR^n)} \,\diffd \bx\right)^{1/2}\\
			\lesssim & \sum_{\bgamma\in \{0,1,2\}^d}\| f^{(\bgamma)} \|_{L_2(\bbR^d)}
			\lesssim \|f\|_{\Hmix^2(\bbR^d)}.
		\end{align*}
		So for each $k \in \bbZ^d$ and $f \in S(\bbR^d)$ we have
		\begin{align*}
			\prod\limits_{i=1}^d (1+|k_i|)^2\left(\sum\limits_{\bm \in\bbZ^d}|{[\mathcal{R}f(\cdot - \bm)]}^\wedge(\bk)|^2\right)^{1/2} 
			\lesssim
			\|f\|_{\Hmix^2(\bbR^d)}.
		\end{align*}
		
		{\em Step 3.} In the last step we use a density argument to generalize the result to $f\in \Hmix^2(\bbR^d)$. Let $(f_n)_{n \in \bbN}\subset S(\bbR^d)$ be a sequence converging to $f$ in $\Hmix^2(\bbR^d)$. In step 2 we have seen that $(f_n)_{n\in \bbN}$ is also a Cauchy sequence in $\bE^2(\bbR^d)$. So there is a $\tilde f \in \bE^2(\bbR^d)$ such that $(f_n)_{n \in \bbN}$ is also converging to $\tilde f$ in $\bE^2(\bbR^d)$. Functions in $\bE^2(\bbR^d)$ do not have to be continuous but $\bE^2(\bbR^d)$ continuously embeds into $L_2(\bbR^d)$. This follows by an application of Parseval's identity. Indeed, for $g \in \bE^2(\bbR^d)$ it holds
		\begin{align*}
			\|g\|^2_{L_2(\bbR^d)} 
			&= \sum_{m \in \bbZ^d}\|g(\cdot-m)\|^2_{L_2([0,1]^d)}
			=\sum_{m \in \bbZ^d}\|\cR g(\cdot-m)\|^2_{L_2([0,1]^d)}\\
			&=\sum_{k \in \bbZ^d}  \prod_{i=1}^d\frac{1}{(1+|k_i|)^4}\prod_{i=1}^d{(1+|k_i|)^4} \sum_{m\in \bbZ^d}|[\cR g(\cdot-m)]^\wedge(k)|^2\\
			&\lesssim \|g\|^2_{\bE^2(\bbR^d)}.
		\end{align*}
		Now we get
		\begin{align*}
			\|f-\tilde f\|_{L_2(\bbR^d)} \lesssim \|f-f_n\|_{\Hmix^2(\bbR^d)}+\|f_n-\tilde f\|_{\bE^2(\bbR^d)}\xrightarrow[n\to \infty]{}0,
		\end{align*}
		that is $f=\tilde f$ almost everywhere and finally 
		\begin{align*}
			\|f\|_{\bE^2(\bbR^d)}=\lim_{n \to \infty}\|f_n\|_{\bE^2(\bbR^d)}\lesssim \lim_{n \to \infty}\|f_n\|_{\Hmix^2(\bbR^d)}=\|f\|_{\Hmix^2(\bbR^d)},
		\end{align*}
		the intended result.
	\end{proof}
	From the previous theorem we directly obtain the following result as a corollary. 
	
	\begin{cor} \label{cor:Korobov}
		The tent transform $\cR$ is a continuous mapping 
		$$
		\cR : \Hmix^2([0,1]^d) \to \bE^2_d.
		$$
	\end{cor}
	
	This should be particularly helpful for the analysis of tent transformed lattice rules in $\Hmix^2([0,1]^d)$. 
	
	\section{Numerical experiments} \label{sec:numerics}
	
	We want to demonstrate our results on some numerical examples. The error $\QMC(X, \Hmix^2([0, 1]^d))$ can be calculated exactly using the formulas from Section \ref{subsec:Sobolev}. When using a computer for this task it is important to consider numerical imprecision from rounding due to the limited number of bits in the float format. This is crucial for larger numbers of points $N$, where we will need to increase the accuracy to get sensible results. All implementations were done in Python using the \texttt{decimal}-library.
	
	\subsection{Tent transformed nets in higher dimensions}
	
	\begin{figure}[htb]
		\begin{tikzpicture}[baseline, scale=0.967]
			\begin{axis}[
				font=\footnotesize,
				enlarge x limits=true,
				enlarge y limits=true,
				height=0.6\textwidth,
				grid=major,
				width=\textwidth,
				xmode=log,
				ymode=log,
				xlabel={$N$},
				ylabel={QMC worst case error},
				legend style={legend cell align=right, at={(0.97, 0.97)}},
				legend columns = 1,
				]
				\addplot[blue,mark=square,mark size=2.2pt,mark options={solid}] coordinates {
					(4, 0.20605851125198316)
					(8, 0.13475421781974734)
					(16, 0.06449605015829911)
					(32, 0.007290746506590049)
					(64, 0.0030031907057914108)
					(128, 0.0012823174647155304)
					(256, 0.0006707468865474971)
					(512, 0.0011607207486715857)
					(1024, 8.699309933592199e-05)
					(2048, 3.543775058010506e-05)
					(4096, 7.871731066909725e-06)
					(8192, 1.861129455341594e-06)
					(16384, 7.853685623460939e-07)
				};
				\addlegendentry{$K_1^3$}
				
				\addplot[magenta,mark=o,mark size=2.5pt,mark options={solid}] coordinates {
					(4, 0.27836542131444547)
					(8, 0.18011531253000726)
					(16, 0.09187615894924508)
					(32, 0.008324452211579923)
					(64, 0.003980734702404352)
					(128, 0.0013536074105748967)
					(256, 0.0006760128700089168)
					(512, 0.0031842582254181053)
					(1024, 0.0001786870187782744)
					(2048, 6.556822998180642e-05)
					(4096, 1.2767563484657267e-05)
					(8192, 3.0764360064204873e-06)
					(16384, 1.297247061327232e-06)
				};
				\addlegendentry{$K_2^3$}
				
				\addplot[green,mark=triangle,mark size=3pt,mark options={solid}] coordinates {
					(4, 0.4701831076771109)
					(8, 0.29201774837959765)
					(16, 0.11256500659922493)
					(32, 0.013650583717789986)
					(64, 0.005341746597102096)
					(128, 0.0018584180756665207)
					(256, 0.00091869684951764)
					(512, 0.001965730398560441)
					(1024, 0.00013179325808720187)
					(2048, 4.569387128187373e-05)
					(4096, 9.974433135226907e-06)
					(8192, 2.580588911054557e-06)
					(16384, 9.767575603136475e-07)
				};
				\addlegendentry{$K_3^3$}
				
				\addplot[orange,mark=star,mark size=2.2pt,mark options={solid}] coordinates {
					(4, 0.19703989367397748)
					(8, 0.05487567293942342)
					(16, 0.019386533292167846)
					(32, 0.0047304985632121484)
					(64, 0.0013806325916908377)
					(128, 0.00044855494142254635)
					(256, 0.00011053100620671699)
					(512, 3.077845500254035e-05)
					(1024, 7.594204356170112e-06)
					(2048, 1.859444507483469e-06)
					(4096, 4.6501719168680607e-07)
					(8192, 1.230421337028527e-07)
				};
				\addlegendentry{Untransformed}
				
				\addplot [forget plot,black,domain=30:20000, samples=100, dotted,ultra thick]{0.5*x^(-2)*(ln(x))^(5)};
				
				\addlegendimage{black,dotted,ultra thick}
				\addlegendentry{$\sim (\log N)^{5}/N^2$}
				
				\addplot [forget plot,gray,domain=4:5000, samples=100, dotted,ultra thick]{0.7*x^(-2)*(ln(x))^(1)};
				\addlegendimage{gray,dotted,ultra thick}
				\addlegendentry{$\sim (\log N)/N^2$}
			\end{axis}
		\end{tikzpicture}
		\caption{Behavior of the worst case error for tent transformed order $2$ nets in dimension $d=3$, measured for the kernels $K_j^3$. We compare this to the error for untransformed order $2$ nets, using the kernel $K_1^3$.}
		\label{fig:coparisson_kernels}
	\end{figure}
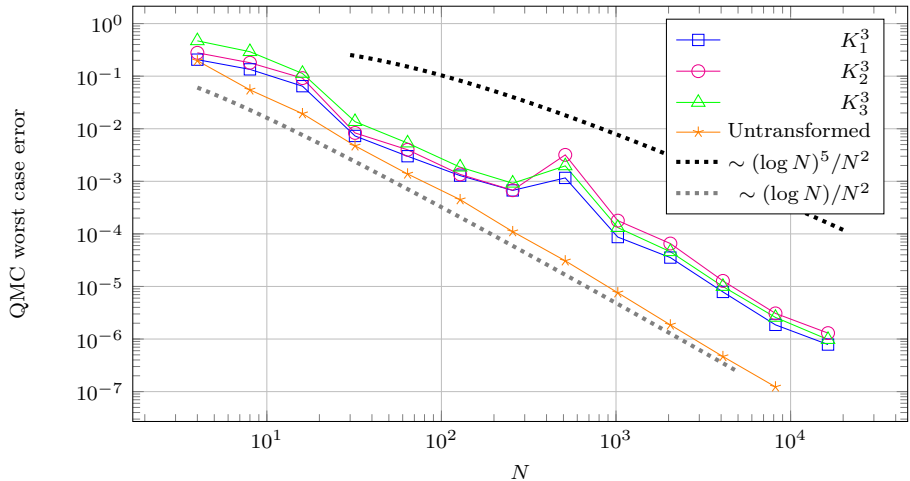 
	
	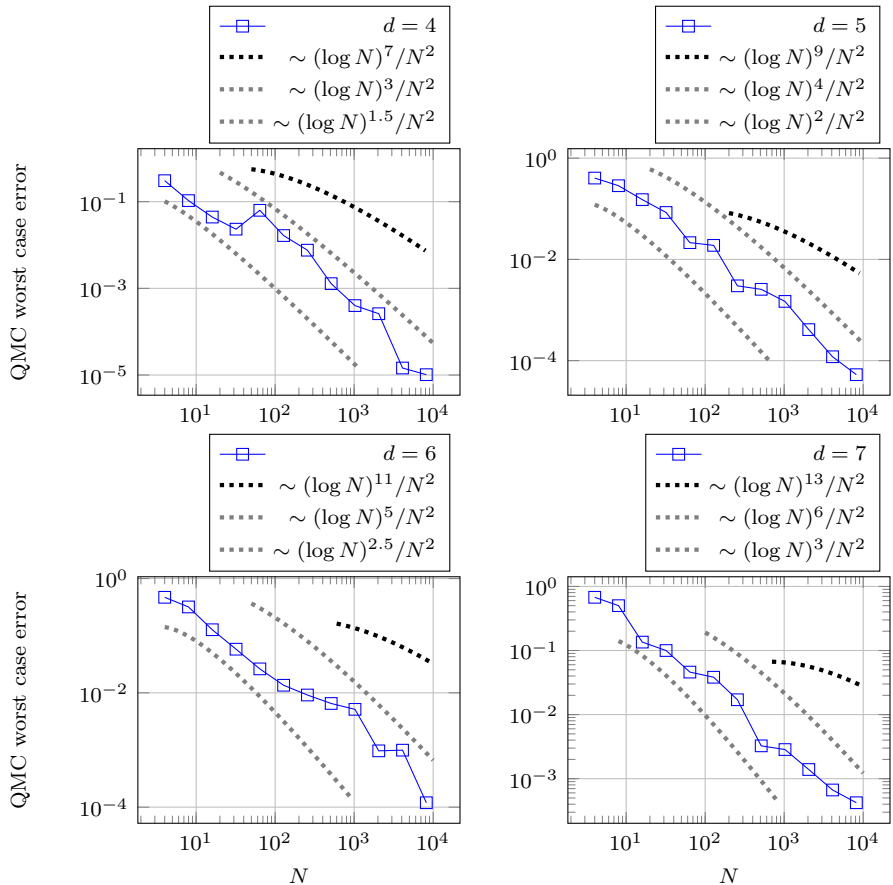
\begin{figure}[htb]
		\centering
		\begin{tikzpicture}[baseline] 
			\begin{axis}[
				font=\footnotesize,
				enlarge x limits=true,
				enlarge y limits=true,
				height=0.4\textwidth,
				grid=major,
				width=0.48\textwidth,
				xmode=log,
				ymode=log,
				ylabel={QMC worst case error},
				legend style={legend cell align=right, at={(0.97, 1.58)}},
				legend columns = 1,
				]
				\addplot[blue,mark=square,mark size=2.2pt,mark options={solid}] coordinates {
					(4, 0.3052461958982308)
					(8, 0.10591592785826713)
					(16, 0.044164306257995445)
					(32, 0.023202907956279616)
					(64, 0.06319987106625677)
					(128, 0.0165360587379965)
					(256, 0.007596940125548403)
					(512, 0.0012840441086573557)
					(1024, 0.0004000741766630895)
					(2048, 0.0002599312168166372)
					(4096, 1.4512468085655088e-05)
					(8192, 1.0198153834635195e-05)
				};
				\addlegendentry{$d=4$}
				\addplot [forget plot,black,domain=50:8000, samples=100, dotted,ultra thick]{0.1*x^(-2)*(ln(x))^(7)};
				\addlegendimage{black,dotted,ultra thick}
				\addlegendentry{$\sim (\log N)^{7}/N^2$}
				\addplot [forget plot,gray,domain=20:10000, samples=100, dotted,ultra thick]{7*x^(-2)*(ln(x))^(3)};
				\addlegendimage{gray,dotted,ultra thick}
				\addlegendentry{$\sim (\log N)^{3}/N^2$}
				\addplot [forget plot,gray,domain=4:1100, samples=100, dotted,ultra thick]{x^(-2)*(ln(x))^(1.5)};
				\addlegendimage{gray,dotted,ultra thick}
				\addlegendentry{$\sim (\log N)^{1.5}/N^2$}
			\end{axis}
		\end{tikzpicture}
		\begin{tikzpicture}[baseline] 
			\begin{axis}[
				font=\footnotesize,
				enlarge x limits=true,
				enlarge y limits=true,
				height=0.4\textwidth,
				grid=major,
				width=0.48\textwidth,
				xmode=log,
				ymode=log,
				legend style={legend cell align=right, at={(0.97, 1.58)}},
				legend columns = 1,
				]
				\addplot[blue,mark=square,mark size=2.2pt,mark options={solid}] coordinates {
					(4, 0.4040210276596287)
					(8, 0.2853696866553136)
					(16, 0.1515898880496842)
					(32, 0.08445346485991725)
					(64, 0.02141253351905513)
					(128, 0.018928585392465943)
					(256, 0.0030051393948116695)
					(512, 0.002562528133670654)
					(1024, 0.0014796543716086921)
					(2048, 0.00041004667346681705)
					(4096, 0.00011942476112681603)
					(8192, 5.2882148481344203e-05)
				};
				\addlegendentry{$d=5$}
				\addplot [forget plot,black,domain=200:9000, samples=100, dotted,ultra thick]{0.001*x^(-2)*(ln(x))^(9)};
				\addlegendimage{black,dotted,ultra thick}
				\addlegendentry{$\sim (\log N)^{9}/N^2$}
				\addplot [forget plot,gray,domain=20:10000, samples=100, dotted,ultra thick]{3*x^(-2)*(ln(x))^(4)};
				\addlegendimage{gray,dotted,ultra thick}
				\addlegendentry{$\sim (\log N)^{4}/N^2$}
				\addplot [forget plot,gray,domain=4:700, samples=100, dotted,ultra thick]{x^(-2)*(ln(x))^(2)};
				\addlegendimage{gray,dotted,ultra thick}
				\addlegendentry{$\sim (\log N)^{2}/N^2$}
				
			\end{axis}
		\end{tikzpicture}
		\begin{tikzpicture}[baseline] 
			\begin{axis}[
				font=\footnotesize,
				enlarge x limits=true,
				enlarge y limits=true,
				height=0.4\textwidth,
				grid=major,
				width=0.48\textwidth,
				xmode=log,
				ymode=log,
				xlabel={$N$},
				ylabel={QMC worst case error},
				legend style={legend cell align=right, at={(0.97, 1.58)}},
				legend columns = 1,
				]
				\addplot[blue,mark=square,mark size=2.2pt,mark options={solid}] coordinates {
					(4, 0.47001787531571904)
					(8, 0.31868551376625637)
					(16, 0.12724157711951184)
					(32, 0.058053770913317704)
					(64, 0.02626196579073494)
					(128, 0.01357689520451767)
					(256, 0.00916651506339192)
					(512, 0.006516209439708014)
					(1024, 0.00514293168373351)
					(2048, 0.0009666593400009463)
					(4096, 0.0009932525580204218)
					(8192, 0.00011948316619382709)
				};
				\addlegendentry{$d=6$}
				\addplot [forget plot,black,domain=600:10000, samples=100, dotted,ultra thick]{0.00008*x^(-2)*(ln(x))^(11)};
				\addlegendimage{black,dotted,ultra thick}
				\addlegendentry{$\sim (\log N)^{11}/N^2$}
				\addplot [forget plot,gray,domain=50:10000, samples=100, dotted,ultra thick]{1*x^(-2)*(ln(x))^(5)};
				\addlegendimage{gray,dotted,ultra thick}
				\addlegendentry{$\sim (\log N)^{5}/N^2$}
				\addplot [forget plot,gray,domain=4:900, samples=100, dotted,ultra thick]{x^(-2)*(ln(x))^(2.5)};
				\addlegendimage{gray,dotted,ultra thick}
				\addlegendentry{$\sim (\log N)^{2.5}/N^2$}
			\end{axis}
		\end{tikzpicture}
		\begin{tikzpicture}[baseline] 
			\begin{axis}[
				font=\footnotesize,
				enlarge x limits=true,
				enlarge y limits=true,
				height=0.4\textwidth,
				grid=major,
				width=0.48\textwidth,
				xmode=log,
				ymode=log,
				xlabel={$N$},
				legend style={legend cell align=right, at={(0.97, 1.58)}},
				legend columns = 1,
				]
				\addplot[blue,mark=square,mark size=2.2pt,mark options={solid}] coordinates {
					(4, 0.676391240139576)
					(8, 0.5009737073493481)
					(16, 0.13517524145356824)
					(32, 0.0999235689006232)
					(64, 0.046046840530074955)
					(128, 0.038270697747904064)
					(256, 0.017052458270655862)
					(512, 0.003253202559627184)
					(1024, 0.0028477241725251997)
					(2048, 0.001389735793673517)
					(4096, 0.0006665480548945198)
					(8192, 0.000421511176808294)
				};
				\addlegendentry{$d=7$}
				\addplot [forget plot,black,domain=700:9000, samples=100, dotted,ultra thick]{0.0000008*x^(-2)*(ln(x))^(13)};
				\addlegendimage{black,dotted,ultra thick}
				\addlegendentry{$\sim (\log N)^{13}/N^2$}
				\addplot [forget plot,gray,domain=100:10000, samples=100, dotted,ultra thick]{0.2*x^(-2)*(ln(x))^(6)};
				\addlegendimage{gray,dotted,ultra thick}
				\addlegendentry{$\sim (\log N)^{6}/N^2$}
				\addplot [forget plot,gray,domain=8:800, samples=100, dotted,ultra thick]{x^(-2)*(ln(x))^(3)};
				\addlegendimage{gray,dotted,ultra thick}
				\addlegendentry{$\sim (\log N)^{3}/N^2$}
			\end{axis}
		\end{tikzpicture}
		\caption{Dimension $d=4, 5, 6, 7$, kernel $K_1^d$.}
		\label{fig:higher_dim}
	\end{figure} 
	
	Let us first compare the errors given by the kernels $K^d_j, j=1, 2, 3$. Since the implied norms are equivalent, all three curves should follow a similar trend. This is visible from the plots in Figure \ref{fig:coparisson_kernels}, where we used the tent transformed order $2$ nets as discussed in Section \ref{sec:nets} in $d=3$ dimensions. We also compare these plots to the upper bound of $(\log N)^5/N^2$ and the optimal lower bound given by $(\log N)/N^2$. It would be interesting to investigate which bound is closer to the true decay of the plotted curves. For the remainder we will only look at the error calculated from $K_1^d$.
	
	Next, we will look at the picture in higher dimensions. In Figure \ref{fig:higher_dim} we plotted the worst case error in dimensions $d=4, 5, 6, 7$ (note that order $5$ nets in dimension $7$ cannot be constructed using the library \cite{KN2016} as described in Section \ref{subsec:construction} since it lists Niederreiter--Xing constructions only up to $33$ dimensions). We see that the plots stay, with some fluctuations, between the predicted bounds. For the upper bound we note that due to the relatively high power of the logarithm we are most likely still in the preasymptotic regime. It is likely that the bound can be reduced considerably and it might be be possible to show something along the lines of $N^{-2} (\log N)^{d-1}$ for tent transformed order $2$ nets. We also plotted these curves, although another exponent for the logarithmic term could also be possible.
	
	\subsection{Halton, Fibonacci, Zaremba in $d=2$}
	
	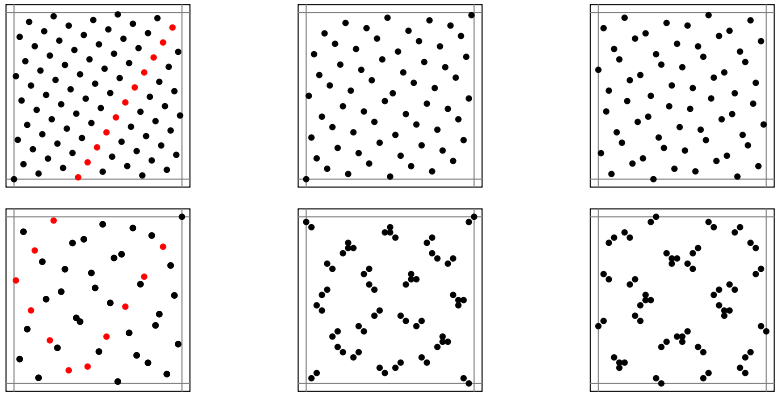
\begin{figure}[htb]
		\centering
		\begin{tikzpicture}
			\begin{axis}[
				plot box ratio = 1 1,
				xmin=-1/21,xmax=22/21,ymin=-1/21,ymax=22/21,
				font=\footnotesize,
				height=0.33\textwidth,
				width=0.33\textwidth,
				xticklabel = \empty,
				yticklabel = \empty,
				ytick style={draw=none},
				xtick style={draw=none},
				]
				\addplot[only marks,black,mark=*,mark size=1pt,mark options={solid}] coordinates { 
					(0/89, 0/89) (1/89, 55/89) (2/89, 21/89) (3/89, 76/89) (4/89, 42/89) (5/89, 8/89) (6/89, 63/89) (7/89, 29/89) (8/89, 84/89) (9/89, 50/89) (10/89, 16/89) (11/89, 71/89) (12/89, 37/89) (13/89, 3/89) (14/89, 58/89) (15/89, 24/89) (16/89, 79/89) (17/89, 45/89) (18/89, 11/89) (19/89, 66/89) (20/89, 32/89) (21/89, 87/89) (22/89, 53/89) (23/89, 19/89) (24/89, 74/89) (25/89, 40/89) (26/89, 6/89) (27/89, 61/89) (28/89, 27/89) (29/89, 82/89) (30/89, 48/89) (31/89, 14/89) (32/89, 69/89) (33/89, 35/89) (35/89, 56/89) (36/89, 22/89) (37/89, 77/89) (38/89, 43/89) (40/89, 64/89) (41/89, 30/89) (42/89, 85/89) (43/89, 51/89) (45/89, 72/89) (46/89, 38/89) (47/89, 4/89) (48/89, 59/89) (50/89, 80/89) (51/89, 46/89) (52/89, 12/89) (53/89, 67/89) (55/89, 88/89) (56/89, 54/89) (57/89, 20/89) (58/89, 75/89) (60/89, 7/89) (61/89, 62/89) (62/89, 28/89) (63/89, 83/89) (65/89, 15/89) (66/89, 70/89) (67/89, 36/89) (68/89, 2/89) (70/89, 23/89) (71/89, 78/89) (72/89, 44/89) (73/89, 10/89) (75/89, 31/89) (76/89, 86/89) (77/89, 52/89) (78/89, 18/89) (80/89, 39/89) (81/89, 5/89) (82/89, 60/89) (83/89, 26/89) (85/89, 47/89) (86/89, 13/89) (87/89, 68/89) (88/89, 34/89)
				};
				\addplot[only marks,red,mark=*,mark size=1pt,mark options={solid}] coordinates { 
					(34/89, 1/89) (39/89, 9/89) (44/89, 17/89) (49/89, 25/89) (54/89, 33/89) (59/89, 41/89) (64/89, 49/89) (69/89, 57/89) (74/89, 65/89) (79/89, 73/89) (84/89, 81/89)
				};
				\addplot[gray,mark options={solid}] coordinates {
					(0, -0.05) (0, 1.05)
				};
				\addplot[gray,mark options={solid}] coordinates {
					(1.0, -0.05) (1.0, 1.05)
				};
				\addplot[gray,mark options={solid}] coordinates {
					(-0.05, 0) (1.05, 0)
				};
				\addplot[gray,mark options={solid}] coordinates {
					(-0.05, 1.0) (1.05, 1.0)
				};
			\end{axis}
		\end{tikzpicture}
		\begin{tikzpicture}
			\begin{axis}[
				plot box ratio = 1 1,
				xmin=-1/21,xmax=22/21,ymin=-1/21,ymax=22/21,
				font=\footnotesize,
				height=0.33\textwidth,
				width=0.33\textwidth,
				xticklabel = \empty,
				yticklabel = \empty,
				ytick style={draw=none},
				xtick style={draw=none},
				]
				\addplot[only marks,black,mark=*,mark size=1pt,mark options={solid}] coordinates {
					(0/64, 0/64) (1/64, 32/64) (2/64, 16/64) (3/64, 48/64) (4/64, 8/64) (5/64, 40/64) (6/64, 24/64) (7/64, 56/64) (8/64, 4/64) (9/64, 36/64) (10/64, 20/64) (11/64, 52/64) (12/64, 12/64) (13/64, 44/64) (14/64, 28/64) (15/64, 60/64) (16/64, 2/64) (17/64, 34/64) (18/64, 18/64) (19/64, 50/64) (20/64, 10/64) (21/64, 42/64) (22/64, 26/64) (23/64, 58/64) (24/64, 6/64) (25/64, 38/64) (26/64, 22/64) (27/64, 54/64) (28/64, 14/64) (29/64, 46/64) (30/64, 30/64) (31/64, 62/64) (32/64, 1/64) (33/64, 33/64) (34/64, 17/64) (35/64, 49/64) (36/64, 9/64) (37/64, 41/64) (38/64, 25/64) (39/64, 57/64) (40/64, 5/64) (41/64, 37/64) (42/64, 21/64) (43/64, 53/64) (44/64, 13/64) (45/64, 45/64) (46/64, 29/64) (47/64, 61/64) (48/64, 3/64) (49/64, 35/64) (50/64, 19/64) (51/64, 51/64) (52/64, 11/64) (53/64, 43/64) (54/64, 27/64) (55/64, 59/64) (56/64, 7/64) (57/64, 39/64) (58/64, 23/64) (59/64, 55/64) (60/64, 15/64) (61/64, 47/64) (62/64, 31/64) (63/64, 63/64)
				};
				\addplot[gray,mark options={solid}] coordinates {
					(0, -0.05) (0, 1.05)
				};
				\addplot[gray,mark options={solid}] coordinates {
					(1.0, -0.05) (1.0, 1.05)
				};
				\addplot[gray,mark options={solid}] coordinates {
					(-0.05, 0) (1.05, 0)
				};
				\addplot[gray,mark options={solid}] coordinates {
					(-0.05, 1.0) (1.05, 1.0)
				};
			\end{axis}
		\end{tikzpicture}
		\begin{tikzpicture}
			\begin{axis}[
				plot box ratio = 1 1,
				xmin=-1/21,xmax=22/21,ymin=-1/21,ymax=22/21,
				font=\footnotesize,
				height=0.33\textwidth,
				width=0.33\textwidth,
				xticklabel = \empty,
				yticklabel = \empty,
				ytick style={draw=none},
				xtick style={draw=none},
				]
				\addplot[only marks,black,mark=*,mark size=1pt,mark options={solid}] coordinates {
					(0/64, 42/64) (1/64, 10/64) (2/64, 58/64) (3/64, 26/64) (4/64, 34/64) (5/64, 2/64) (6/64, 50/64) (7/64, 18/64) (8/64, 46/64) (9/64, 14/64) (10/64, 62/64) (11/64, 30/64) (12/64, 38/64) (13/64, 6/64) (14/64, 54/64) (15/64, 22/64) (16/64, 40/64) (17/64, 8/64) (18/64, 56/64) (19/64, 24/64) (20/64, 32/64) (21/64, 0/64) (22/64, 48/64) (23/64, 16/64) (24/64, 44/64) (25/64, 12/64) (26/64, 60/64) (27/64, 28/64) (28/64, 36/64) (29/64, 4/64) (30/64, 52/64) (31/64, 20/64) (32/64, 43/64) (33/64, 11/64) (34/64, 59/64) (35/64, 27/64) (36/64, 35/64) (37/64, 3/64) (38/64, 51/64) (39/64, 19/64) (40/64, 47/64) (41/64, 15/64) (42/64, 63/64) (43/64, 31/64) (44/64, 39/64) (45/64, 7/64) (46/64, 55/64) (47/64, 23/64) (48/64, 41/64) (49/64, 9/64) (50/64, 57/64) (51/64, 25/64) (52/64, 33/64) (53/64, 1/64) (54/64, 49/64) (55/64, 17/64) (56/64, 45/64) (57/64, 13/64) (58/64, 61/64) (59/64, 29/64) (60/64, 37/64) (61/64, 5/64) (62/64, 53/64) (63/64, 21/64) 
				};
				\addplot[gray,mark options={solid}] coordinates {
					(0, -0.05) (0, 1.05)
				};
				\addplot[gray,mark options={solid}] coordinates {
					(1.0, -0.05) (1.0, 1.05)
				};
				\addplot[gray,mark options={solid}] coordinates {
					(-0.05, 0) (1.05, 0)
				};
				\addplot[gray,mark options={solid}] coordinates {
					(-0.05, 1.0) (1.05, 1.0)
				};
			\end{axis}
		\end{tikzpicture}
		
		\begin{tikzpicture}
			\begin{axis}[
				plot box ratio = 1 1,
				xmin=-1/21,xmax=22/21,ymin=-1/21,ymax=22/21,
				font=\footnotesize,
				height=0.33\textwidth,
				width=0.33\textwidth,
				xticklabel = \empty,
				yticklabel = \empty,
				ytick style={draw=none},
				xtick style={draw=none},
				]
				\addplot[only marks,black,mark=*,mark size=1pt,mark options={solid}] coordinates { 
					(89/89, 89/89) (87/89, 21/89) (85/89, 47/89) (83/89, 63/89) (81/89, 5/89) (77/89, 37/89) (75/89, 31/89) (73/89, 79/89) (71/89, 11/89) (67/89, 53/89) (65/89, 15/89) (63/89, 83/89) (61/89, 27/89) (57/89, 69/89) (55/89, 1/89) (53/89, 67/89) (51/89, 43/89) (47/89, 85/89) (45/89, 17/89) (43/89, 51/89) (41/89, 59/89) (37/89, 77/89) (35/89, 33/89) (33/89, 35/89) (31/89, 75/89) (27/89, 61/89) (25/89, 49/89) (23/89, 19/89) (17/89, 45/89) (15/89, 65/89) (13/89, 3/89) (7/89, 29/89) (5/89, 81/89) (3/89, 13/89) (3/89, 13/89) (5/89, 81/89) (7/89, 29/89) (13/89, 3/89) (15/89, 65/89) (17/89, 45/89) (23/89, 19/89) (25/89, 49/89) (27/89, 61/89) (31/89, 75/89) (33/89, 35/89) (35/89, 33/89) (37/89, 77/89) (41/89, 59/89) (43/89, 51/89) (45/89, 17/89) (47/89, 85/89) (51/89, 43/89) (53/89, 67/89) (55/89, 1/89) (57/89, 69/89) (61/89, 27/89) (63/89, 83/89) (65/89, 15/89) (67/89, 53/89) (71/89, 11/89) (73/89, 79/89) (75/89, 31/89) (77/89, 37/89) (81/89, 5/89) (83/89, 63/89) (85/89, 47/89) (87/89, 21/89)
				};
				\addplot[only marks,red,mark=*,mark size=1pt,mark options={solid}] coordinates { 
					(21/89, 87/89) (11/89, 71/89) (1/89, 55/89) (9/89, 39/89) (19/89, 23/89) (29/89, 7/89) (39/89, 9/89) (49/89, 25/89) (59/89, 41/89) (69/89, 57/89) (79/89, 73/89)
				};
				\addplot[gray,mark options={solid}] coordinates {
					(0, -0.05) (0, 1.05)
				};
				\addplot[gray,mark options={solid}] coordinates {
					(1.0, -0.05) (1.0, 1.05)
				};
				\addplot[gray,mark options={solid}] coordinates {
					(-0.05, 0) (1.05, 0)
				};
				\addplot[gray,mark options={solid}] coordinates {
					(-0.05, 1.0) (1.05, 1.0)
				};
			\end{axis}
		\end{tikzpicture}
		\begin{tikzpicture}
			\begin{axis}[
				plot box ratio = 1 1,
				xmin=-1/21,xmax=22/21,ymin=-1/21,ymax=22/21,
				font=\footnotesize,
				height=0.33\textwidth,
				width=0.33\textwidth,
				xticklabel = \empty,
				yticklabel = \empty,
				ytick style={draw=none},
				xtick style={draw=none},
				]
				\addplot[only marks,black,mark=*,mark size=1pt,mark options={solid}] coordinates {
					(64/64, 64/64) (62/64, 0/64) (60/64, 32/64) (58/64, 32/64) (56/64, 48/64) (54/64, 16/64) (52/64, 16/64) (50/64, 48/64) (48/64, 56/64) (46/64, 8/64) (44/64, 24/64) (42/64, 40/64) (40/64, 40/64) (38/64, 24/64) (36/64, 8/64) (34/64, 56/64) (32/64, 60/64) (30/64, 4/64) (28/64, 28/64) (26/64, 36/64) (24/64, 44/64) (22/64, 20/64) (20/64, 12/64) (18/64, 52/64) (16/64, 52/64) (14/64, 12/64) (12/64, 20/64) (10/64, 44/64) (8/64, 36/64) (6/64, 28/64) (4/64, 4/64) (2/64, 60/64) (0/64, 62/64) (2/64, 2/64) (4/64, 30/64) (6/64, 34/64) (8/64, 46/64) (10/64, 18/64) (12/64, 14/64) (14/64, 50/64) (16/64, 54/64) (18/64, 10/64) (20/64, 22/64) (22/64, 42/64) (24/64, 38/64) (26/64, 26/64) (28/64, 6/64) (30/64, 58/64) (32/64, 58/64) (34/64, 6/64) (36/64, 26/64) (38/64, 38/64) (40/64, 42/64) (42/64, 22/64) (44/64, 10/64) (46/64, 54/64) (48/64, 50/64) (50/64, 14/64) (52/64, 18/64) (54/64, 46/64) (56/64, 34/64) (58/64, 30/64) (60/64, 2/64) (62/64, 62/64) 
				};
				\addplot[gray,mark options={solid}] coordinates {
					(0, -0.05) (0, 1.05)
				};
				\addplot[gray,mark options={solid}] coordinates {
					(1.0, -0.05) (1.0, 1.05)
				};
				\addplot[gray,mark options={solid}] coordinates {
					(-0.05, 0) (1.05, 0)
				};
				\addplot[gray,mark options={solid}] coordinates {
					(-0.05, 1.0) (1.05, 1.0)
				};
			\end{axis}
		\end{tikzpicture}
		\begin{tikzpicture}
			\begin{axis}[
				plot box ratio = 1 1,
				xmin=-1/21,xmax=22/21,ymin=-1/21,ymax=22/21,
				font=\footnotesize,
				height=0.33\textwidth,
				width=0.33\textwidth,
				xticklabel = \empty,
				yticklabel = \empty,
				ytick style={draw=none},
				xtick style={draw=none},
				]
				\addplot[only marks,black,mark=*,mark size=1pt,mark options={solid}] coordinates {
					(64/64, 20/64) (62/64, 44/64) (60/64, 52/64) (58/64, 12/64) (56/64, 4/64) (54/64, 60/64) (52/64, 36/64) (50/64, 28/64) (48/64, 28/64) (46/64, 36/64) (44/64, 60/64) (42/64, 4/64) (40/64, 12/64) (38/64, 52/64) (36/64, 44/64) (34/64, 20/64) (32/64, 16/64) (30/64, 48/64) (28/64, 48/64) (26/64, 16/64) (24/64, 0/64) (22/64, 64/64) (20/64, 32/64) (18/64, 32/64) (16/64, 24/64) (14/64, 40/64) (12/64, 56/64) (10/64, 8/64) (8/64, 8/64) (6/64, 56/64) (4/64, 40/64) (2/64, 24/64) (0/64, 22/64) (2/64, 42/64) (4/64, 54/64) (6/64, 10/64) (8/64, 6/64) (10/64, 58/64) (12/64, 38/64) (14/64, 26/64) (16/64, 30/64) (18/64, 34/64) (20/64, 62/64) (22/64, 2/64) (24/64, 14/64) (26/64, 50/64) (28/64, 46/64) (30/64, 18/64) (32/64, 18/64) (34/64, 46/64) (36/64, 50/64) (38/64, 14/64) (40/64, 2/64) (42/64, 62/64) (44/64, 34/64) (46/64, 30/64) (48/64, 26/64) (50/64, 38/64) (52/64, 58/64) (54/64, 6/64) (56/64, 10/64) (58/64, 54/64) (60/64, 42/64) (62/64, 22/64)  
				};
				\addplot[gray,mark options={solid}] coordinates {
					(0, -0.05) (0, 1.05)
				};
				\addplot[gray,mark options={solid}] coordinates {
					(1.0, -0.05) (1.0, 1.05)
				};
				\addplot[gray,mark options={solid}] coordinates {
					(-0.05, 0) (1.05, 0)
				};
				\addplot[gray,mark options={solid}] coordinates {
					(-0.05, 1.0) (1.05, 1.0)
				};
			\end{axis}
		\end{tikzpicture}
		\caption{The considered point sets for $d=2$. From left to right: Fibonacci ($N = 89$), Halton ($N=64$), Zaremba ($N = 64$). Top row: untransformed ($Y \sbse \bbT^2$). Bottom row: tent transformed ($X = \cR Y \sbse [0, 1]^2$). Some points are colored red to show how they map under the tent transform.}
		\label{fig:plane_point_sets}
	\end{figure} 
	
	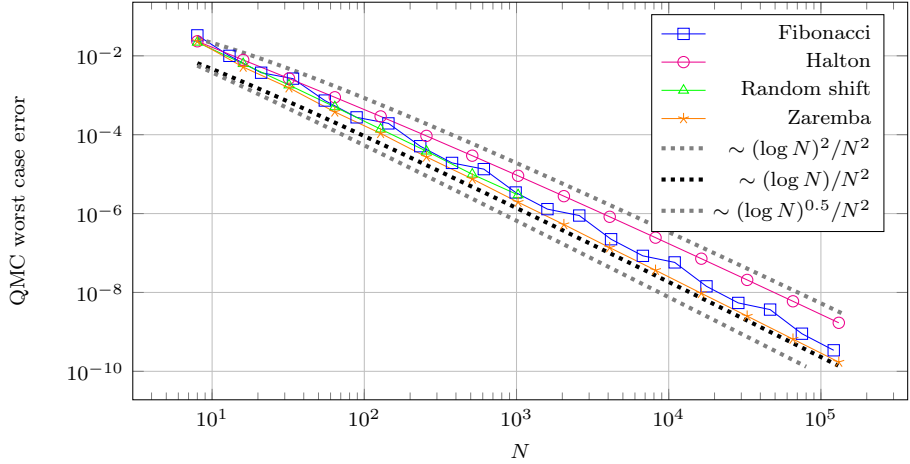
\begin{figure}[htb]
		\begin{tikzpicture}[baseline, scale=0.967]
			\begin{axis}[
				font=\footnotesize,
				enlarge x limits=true,
				enlarge y limits=true,
				height=0.58\textwidth,
				grid=major,
				width=\textwidth,
				xmode=log,
				ymode=log,
				xlabel={$N$},
				ylabel={QMC worst case error},
				legend style={legend cell align=right, at={(0.97, 0.97)}},
				legend columns = 1,
				]
				\addplot[blue,mark=square,mark size=2.2pt,mark options={solid}] coordinates {
					(8, 0.03302638668049939158417690712)
					(13, 0.009977467123244848763750672049)
					(21, 0.003728202477392069525523423106)
					(34, 0.002639228758110435981997928028)
					(55, 0.0007338701865181458071172582848)
					(89, 0.0002750918918299115141351928813)
					(144, 0.0001935872110368731490703700744)
					(233, 0.00005107660256386584154575624214)
					(377, 0.00001924883371081124613495717813)
					(610, 0.00001339066242442836746460185805)
					(987, 0.000003425857239552600212815615374)
					(1597, 0.000001294978704119109056787165912)
					(2584, 8.917637535811213940775062683e-7)
					(4181, 2.235715558561911515687599703e-7)
					(6765, 8.467924270187470211041746534e-8)
					(10946, 5.782149699965403005736672906e-8)
					(17711, 1.429162205566603975746391899e-8)
					(28657, 5.420591731813049946920910228e-9)
					(46368, 3.675624525029726155856346121e-9)
					(75025, 8.990281269237353510488018932e-10)
					(121393, 3.413661934052638435050907592e-10)
				};
				\addlegendentry{Fibonacci}
				
				\addplot[magenta,mark=o,mark size=2.2pt,mark options={solid}] coordinates {
					(8, 0.02336484597622991)
					(16, 0.007864130072342644)
					(32, 0.0026631296907772845)
					(64, 0.0008884917480404264)
					(128, 0.00029073922487487093)
					(256, 9.32704369872836e-05)
					(512, 2.9356790230768647e-05)
					(1024, 0.000009077766563501195584155979303)
					(2048, 0.000002762143118486200381555248794)
					(4096, 8.283606176982401666449943441e-7)
					(8192, 2.4523003357120434e-07)
					(16384, 7.176683770812951e-08)
					(32768, 2.0788199188746484e-08)
					(65536, 5.966681453622273e-09)
					(131072, 1.698602035792963483754218190e-9)
				};
				\addlegendentry{Halton}
				
				\addplot[green,mark=triangle,mark size=2.2pt,mark options={solid}] coordinates {
					(8, 0.02201659153507089341483764516)
					(16, 0.006301298852200637506394107996)
					(32, 0.001813900344770323201740000376)
					(64, 0.0005068271601736927733612984975)
					(128, 0.0001383667353963565227128971385)
					(256, 0.00003820058065403319918054377908)
					(512, 0.000009923272952228823666942131956)
					(1024, 0.000002973907684589775833233206403)
				};
				\addlegendentry{Random shift}
				
				\addplot[orange,mark=star,mark size=2.2pt,mark options={solid}] coordinates {
					(8, 0.02336484597622991069673846458)
					(16, 0.005174903475306061992784737502)
					(32, 0.001545345015799326088310796345)
					(64, 0.0003776557434983195163229354884)
					(128, 0.0001077836849599638568926136977)
					(256, 0.00002726989921941751547843932061)
					(512, 0.000007555812349623088277189542535)
					(1024, 0.000001942436292453704902202823922)
					(2048, 5.271768515742112893631983427e-7)
					(4096, 1.366484567658486218588371331e-7)
					(8192, 3.653010150630572775660420486e-8)
					(16384, 9.509427279915441889787123433e-9)
					(32768, 2.513492793106835227951374960e-9)
					(65536, 6.556179199808376195693619869e-10)
					(131072, 1.717586387929294227730861795e-10)
				};
				\addlegendentry{Zaremba}
				
				
				\addplot [forget plot,gray,domain=8:140000, samples=100, dotted,ultra thick]{0.4*x^(-2)*ln(x)^2};
				\addlegendimage{gray,dotted,ultra thick}
				\addlegendentry{$\sim (\log N)^2/N^2$}
				
				\addplot [forget plot,black,domain=8:130000, samples=100, dotted,ultra thick]{0.2*x^(-2)*ln(x)};
				\addlegendimage{black,dotted,ultra thick}
				\addlegendentry{$\sim (\log N)/N^2$}
				
				\addplot [forget plot,gray,domain=8:80000, samples=100, dotted,ultra thick]{0.25*x^(-2)*ln(x)^(0.5)};
				\addlegendimage{gray,dotted,ultra thick}
				\addlegendentry{$\sim (\log N)^{0.5}/N^2$}
			\end{axis}
		\end{tikzpicture}
		\caption{Dimension $d=2$, kernel $K_1^2$, tent transformed Fibonacci lattices, Halton point sets (in base $2$), random digital shifts of Halton point sets and Zaremba point sets.}
		\label{fig:fib_halt_zar}
	\end{figure} 
	
	For comparison we consider three common constructions in the case $d=2$. Figure \ref{fig:plane_point_sets} shows the Fibonacci lattice, Halton points and Zaremba points, untransformed in the torus and then tent transformed into the square. Figure \ref{fig:fib_halt_zar} shows the corresponding worst case errors. Let us first look at the tent transformed Fibonacci lattice (which can be a multiset since two different points of the lattice can be transformed to the same point). We observe a close fit to the rate of $(\log N)/N^2$. A corresponding upper bound can be derived from Corollary \ref{cor:Korobov} together with the optimality of the Fibonacci lattice in $\Hmix^1(\bbT^d)$, see \cite[Section 8.4]{DTU18} and the references therein. Tent transformed lattices in higher dimensions have been  investigated in more detail in \cite{CoKuNuSu16, DGS2022}. 
	
	For comparison let us look at tent transformed Halton point set of size $N=2^n$ generated via binary digit reversals. While we see that the Halton points are generally worse then the Fibonacci lattice, the rates only seem to differ in the power of the logarithm. It could even be possible to get a closed form expression for the corresponding worst case error (similar to \cite[Theorem 8]{HKP2021}) where the dominating term is given by $n^2/N^2$.
	
	We also included some data on the \emph{digitally shifted} Halton sets (again see \cite{HKP2021}, the discussion around Theorem 9), which visibly improves the quality of the point set. We consider point sets via random digital shifts, where we calculated these values empirically, averaging the worst case error of $200$ randomly generated digitally shifted Halton point sets. Perhaps it is also possible to give a closed form expression for the average case error with dominating term $n/N^2$ similar to \cite{CDP06, JosefDick2005, HKP2021}.
	
	In \cite{HZ1969} (also see \cite{FP2011}) it was observed that a concrete digital shift, namely changing the parity of every second bit in the $y$-coordinate, improves the quality of the Halton point sets. The plot for these \emph{Zaremba point sets} indicates a similar improvement for our situation. We did not calculate the precise asymptotic behavior of these point sets, although it does not seem to improve too much on the random digital shifts. It would be interesting to investigate if one can give an explicit construction of a (order $2$) digital net achieving the optimal rate of \eqref{eq:GSY_optimal_order}.
	
	\begin{figure}[htb]
		\begin{tikzpicture}[baseline, scale=0.967]
			\begin{axis}[
				font=\footnotesize,
				enlarge x limits=true,
				enlarge y limits=true,
				height=0.5\textwidth,
				grid=major,
				width=\textwidth,
				xmode=log,
				ymode=log,
				xlabel={$N$},
				ylabel={QMC error},
				legend style={legend cell align=right, at={(0.97, 0.97)}},
				legend columns = 1,
				]
				\addplot [forget plot,black,domain=400:600000, samples=100, dotted,ultra thick]{3*ln(x)^5*x^(-2)};
				\addlegendimage{black,dotted,ultra thick}
				\addlegendentry{$\sim (\log N)^5/N^2$}
				
				\addplot[blue,mark=square,mark size=2.2pt,mark options={solid}] coordinates {
					(4, 0.05419342175969326)
					(8, 0.03450510395342172)
					(16, 0.009842047891461328)
					(32, 0.012437191659413364)
					(64, 0.0043206323927829195)
					(128, 0.006503642514446228)
					(256, 0.001361159190235231)
					(512, 0.0004982987213616745)
					(1024, 7.109405058150175e-05)
					(2048, 4.2109371703063985e-05)
					(4096, 0.00010180191760824226)
					(8192, 0.00010221017588662982)
					(16384, 3.450261582669255e-05)
					(32768, 1.619363383164572e-05)
					(65536, 6.538407201589883e-06)
					(131072, 4.103956829783826e-07)
					(262144, 6.294633606827811e-07)
					(524288, 5.811858692539296e-08)
				};
				
				\addplot[orange,mark=star,mark size=2.2pt,mark options={solid}] coordinates {
					(4, 0.10640468275012939)
					(8, 0.024159222328965096)
					(16, 0.015050444988515994)
					(32, 0.0009828742006763706)
					(64, 0.0014070748890046756)
					(128, 0.0014459247536425912)
					(256, 0.000973119840648763)
					(512, 1.741075543159949e-05)
					(1024, 6.754788048014342e-05)
					(2048, 0.00034606183534524685)
					(4096, 9.374563199466181e-05)
					(8192, 1.6559815602427973e-05)
					(16384, 4.955177150058794e-05)
					(32768, 2.882564575248376e-06)
					(65536, 7.0870991987356356e-06)
					(131072, 4.0981922521287e-06)
					(262144, 8.689343697163583e-07)
					(524288, 7.526238189627788e-08)
				};
				
				\addplot[blue,mark=square,mark size=2.2pt,mark options={solid}] coordinates {
					(4, 0.16318206360481044)
					(8, 0.1128855331650965)
					(16, 0.02284673298649497)
					(32, 0.018547248015082028)
					(64, 0.0031747107417253225)
					(128, 2.488098988505171e-05)
					(256, 0.0018960338311760005)
					(512, 0.0014245201090030567)
					(1024, 3.6065982446990244e-05)
					(2048, 0.0004757062256527166)
					(4096, 0.00010472550908747171)
					(8192, 0.00016751098269591894)
					(16384, 6.676335681944987e-05)
					(32768, 6.0572772831709824e-05)
					(65536, 4.3845541661945816e-06)
					(131072, 1.0644646459649031e-05)
					(262144, 1.2965745740614744e-06)
					(524288, 1.0037428637462133e-07)
				};
				
				\addplot[blue,mark=square,mark size=2.2pt,mark options={solid}] coordinates {
					(4, 0.011222783402922052)
					(8, 0.005517720316833033)
					(16, 0.04482425183589085)
					(32, 0.022147074268487943)
					(64, 0.012865730013275147)
					(128, 0.011292004575437808)
					(256, 0.0011505501557724975)
					(512, 0.0027875247068005157)
					(1024, 0.0008983089592524123)
					(2048, 0.00013405300889731428)
					(4096, 0.00036266487314089553)
					(8192, 3.4031084313530373e-05)
					(16384, 1.1466033667484196e-05)
					(32768, 4.11066664804379e-05)
					(65536, 5.502889889781902e-06)
					(131072, 1.9315868867309568e-06)
					(262144, 2.066881548743793e-06)
					(524288, 4.5697300688843115e-07)
				};
				\addlegendentry{Tent transformed}
				
				\addplot[orange,mark=star,mark size=2.2pt,mark options={solid}] coordinates {
					(4, 0.12569339315776623)
					(8, 0.03638567544716477)
					(16, 0.04028667286322672)
					(32, 0.0010394222214119616)
					(64, 0.020583291676830606)
					(128, 0.0065634812678797136)
					(256, 0.002432522712398226)
					(512, 0.0015689689501964357)
					(1024, 0.00022356894587753198)
					(2048, 0.000428741280248517)
					(4096, 0.00011806053991004341)
					(8192, 0.0001072445673887291)
					(16384, 5.288108894006047e-05)
					(32768, 1.7286359760107268e-07)
					(65536, 5.208990214218552e-06)
					(131072, 7.285134564282504e-06)
					(262144, 8.3944049596953e-07)
					(524288, 1.5287514990477996e-07)
				};
				\addplot[orange,mark=star,mark size=2.2pt,mark options={solid}] coordinates {
					(4, 0.06490476339828935)
					(8, 0.034331278795008933)
					(16, 0.03134083386977728)
					(32, 0.010031338970971592)
					(64, 0.01243840216386504)
					(128, 0.0020315671204195066)
					(256, 0.0008150908647325825)
					(512, 0.0006058097046230358)
					(1024, 0.0006721609344417667)
					(2048, 0.0002421484775807414)
					(4096, 3.436839882166441e-05)
					(8192, 7.340577854098275e-05)
					(16384, 4.6209682123280336e-05)
					(32768, 6.106999536483942e-06)
					(65536, 7.7506134924416e-06)
					(131072, 6.353462465725236e-08)
					(262144, 2.272935120406276e-07)
					(524288, 2.3474040832382515e-08)
				};
				\addlegendentry{Untransformed}
				
			\end{axis}
		\end{tikzpicture}
		\caption{Tensor product of piecewise linear functions, $d=3$ dimensions.}
		\label{fig:piecewise_linear}
	\end{figure}
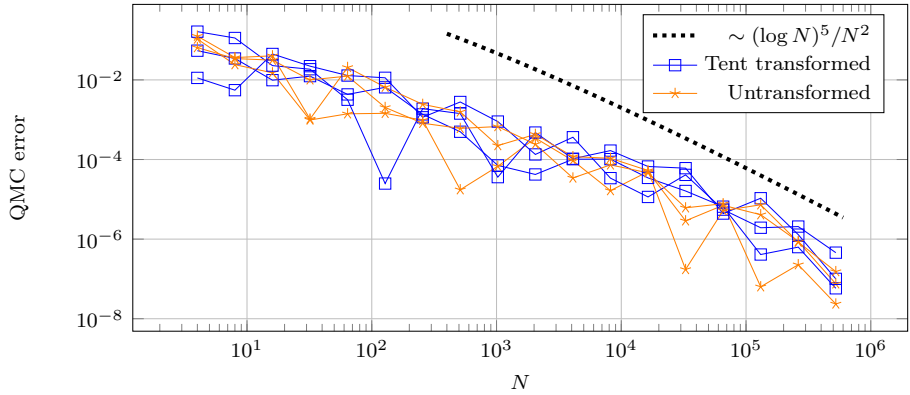 
	
	\begin{figure}[htb]
		\begin{tikzpicture}[baseline, scale=0.967]
			\begin{axis}[
				font=\footnotesize,
				enlarge x limits=true,
				enlarge y limits=true,
				height=0.45\textwidth,
				grid=major,
				width=0.5\textwidth,
				xmode=log,
				ymode=log,
				xlabel={$N$},
				ylabel={QMC error},
				legend style={legend cell align=right, at={(0.97, 1.6)}},
				legend columns = 1,
				]
				\addplot [forget plot,black,domain=400:600000, samples=100, dotted,ultra thick]{0.1*ln(x)^5*x^(-2)};
				\addlegendimage{black,dotted,ultra thick}
				\addlegendentry{$\sim (\log N)^5/N^{2}$}
				
				\addplot[blue,mark=square,mark size=2.2pt,mark options={solid}] coordinates {
					(4, 0.0646789692066334)
					(8, 0.041842929723551844)
					(16, 0.01792306359463101)
					(32, 0.0015582259648913833)
					(64, 0.0003084657005897599)
					(128, 4.7696489707941114e-05)
					(256, 5.4072722379661385e-05)
					(512, 0.00036794167672064105)
					(1024, 1.5189142625527136e-05)
					(2048, 9.829641946002508e-06)
					(4096, 1.04553327025387e-07)
					(8192, 6.827481149626506e-08)
					(16384, 6.94516555827521e-08)
					(32768, 6.759442717950895e-09)
					(65536, 2.7191318523838184e-10)
					(131072, 5.993734166340415e-11)
					(262144, 3.9486190817562514e-11)
					(524288, 3.237969944015853e-13)
				};
				\addlegendentry{Tent transformed}
				\addplot[orange,mark=star,mark size=2.2pt,mark options={solid}] coordinates {
					(4, 0.046049947264017885)
					(8, 0.008396823669281057)
					(16, 0.0014196259584939473)
					(32, 0.002450317420207609)
					(64, 0.0005400486868549593)
					(128, 0.00014343010117212718)
					(256, 2.082341494241382e-05)
					(512, 1.1049985694505109e-05)
					(1024, 4.961552397104355e-07)
					(2048, 1.1515336280328792e-07)
					(4096, 7.355991783618432e-08)
					(8192, 4.45095968638535e-09)
					(16384, 3.5177859193683774e-09)
					(32768, 4.3312355612568116e-10)
					(65536, 1.0609412806370372e-10)
					(131072, 4.3849256904886334e-11)
					(262144, 6.70949543220085e-12)
					(524288, 9.033130286013927e-13)
				};
				\addlegendentry{Untransformed}
				\addplot[green,mark=triangle,mark size=2.2pt,mark options={solid}] coordinates {
					(4, 0.03252776117847976)
					(8, 0.00843447643181045)
					(16, 0.00039175506201682384)
					(32, 0.00035332904418438634)
					(64, 0.0003287736615767048)
					(128, 7.68060416139876e-05)
					(256, 7.075462296545515e-05)
					(512, 1.3037814522303992e-05)
					(1024, 1.9635099808172463e-06)
					(2048, 1.4190323125904297e-07)
					(4096, 9.196670757452359e-08)
					(8192, 1.7261898606543538e-08)
					(16384, 1.254389732076e-10)
					(32768, 8.741951640402113e-10)
					(65536, 1.8025279618851886e-11)
					(131072, 1.0219906664093047e-11)
					(262144, 3.0057141337059104e-13)
					(524288, 4.413305461038136e-15)
				};
				\addlegendentry{Order $3$, untr.}
				\addplot[magenta,mark=o,mark size=2.2pt,mark options={solid}] coordinates {
					(4, 0.02634722657536191)
					(8, 0.03137980116509052)
					(16, 0.0146509596860152)
					(32, 0.016283672970795957)
					(64, 0.016988692918967176)
					(128, 0.01673026453123859)
					(256, 0.0004438733136120094)
					(512, 0.0004322568800541661)
					(1024, 2.794158911091714e-05)
					(2048, 3.215212192793672e-06)
					(4096, 1.1185816419510278e-05)
					(8192, 9.160324199873773e-06)
					(16384, 3.470938763934636e-08)
					(32768, 2.253397720881239e-08)
					(65536, 1.383805851233297e-08)
					(131072, 1.2003103184685255e-09)
					(262144, 3.477164501747486e-10)
					(524288, 1.3556408381026878e-10)
				};
				\addlegendentry{Order $3$, tent tr.}
			\end{axis}
		\end{tikzpicture}
		\begin{tikzpicture}[baseline, scale=0.967]
			\begin{axis}[
				font=\footnotesize,
				enlarge x limits=true,
				enlarge y limits=true,
				height=0.45\textwidth,
				grid=major,
				width=0.5\textwidth,
				xmode=log,
				ymode=log,
				xlabel={$N$},
				ylabel={},
				legend style={legend cell align=right, at={(0.97, 1.6)}},
				legend columns = 1,
				]
				\addplot [forget plot,black,domain=400:600000, samples=100, dotted,ultra thick]{0.0001*ln(x)^9*x^(-2)};
				\addlegendimage{black,dotted,ultra thick}
				\addlegendentry{$\sim (\log N)^9/N^{2}$}
				
				\addplot[blue,mark=square,mark size=2.2pt,mark options={solid}] coordinates {
					(4, 0.0010765495915119171)
					(8, 0.006662738143544055)
					(16, 0.004404665650709674)
					(32, 0.0018111620267278513)
					(64, 0.0015961123058837769)
					(128, 0.0016299702765190508)
					(256, 0.00012122721618308704)
					(512, 0.00017101693592562088)
					(1024, 8.239864079641633e-05)
					(2048, 8.80523630583016e-06)
					(4096, 3.7948293052914488e-06)
					(8192, 7.849908953995014e-06)
					(16384, 5.8502495317154046e-06)
					(32768, 2.5285182599987665e-06)
					(65536, 1.2355994964141889e-06)
					(131072, 2.876138785388449e-08)
					(262144, 1.9915519077546402e-08)
					(524288, 6.734377607823001e-09)
				};
				\addlegendentry{Tent transformed}
				\addplot[orange,mark=star,mark size=2.2pt,mark options={solid}] coordinates {
					(4, 0.05161857687516645)
					(8, 0.018913082322005736)
					(16, 0.010809645892221977)
					(32, 0.004844881273377357)
					(64, 0.00029687499814511667)
					(128, 0.00034916494327445985)
					(256, 0.00020547642211968492)
					(512, 2.3635688955174e-05)
					(1024, 7.284625350874383e-06)
					(2048, 2.8823417473573308e-05)
					(4096, 7.34653433448627e-07)
					(8192, 8.046903697346647e-07)
					(16384, 1.5394705175947557e-07)
					(32768, 1.0923617111391858e-08)
					(65536, 4.48898560427076e-08)
					(131072, 1.8841446072611452e-08)
					(262144, 9.766536533868546e-09)
					(524288, 1.885448358819942e-09)
				};
				\addlegendentry{Untransformed}
				\addplot[green,mark=triangle,mark size=2.2pt,mark options={solid}] coordinates {
					(4, 0.03868285986746312)
					(8, 0.018035385543686642)
					(16, 0.007991535464297596)
					(32, 0.006740167662145605)
					(64, 0.004037029262975665)
					(128, 3.459753171168436e-05)
					(256, 0.00019292817681292862)
					(512, 0.00032828309670747814)
					(1024, 8.6350104642839e-05)
					(2048, 3.9979518523888606e-05)
					(4096, 3.508860278843912e-05)
					(8192, 3.2721070324587625e-05)
					(16384, 2.8136749518508003e-06)
					(32768, 3.7626969852790776e-07)
					(65536, 5.006103180776115e-08)
					(131072, 2.050225206899542e-08)
					(262144, 1.874238734824551e-09)
					(524288, 3.908149440647843e-09)
				};
				\addlegendentry{Order $3$, untr.}
				\addplot[magenta,mark=o,mark size=2.2pt,mark options={solid}] coordinates {
					(4, 0.01816159950015829)
					(8, 0.019724896387236547)
					(16, 0.009438477075073493)
					(32, 0.008796986833159057)
					(64, 0.001485050848774675)
					(128, 0.0017086470709074862)
					(256, 0.0014740387468675138)
					(512, 4.7009350777642e-05)
					(1024, 1.4236550657202152e-05)
					(2048, 2.901180977208286e-06)
					(4096, 9.662775133255152e-06)
					(8192, 2.236728354619276e-06)
					(16384, 3.309358934265996e-06)
					(32768, 9.071258541291102e-06)
					(65536, 7.346354407605977e-06)
					(131072, 7.646514979693586e-07)
					(262144, 1.0935808529642631e-07)
					(524288, 2.064807476106493e-08)
				};
				\addlegendentry{Order $3$, tent tr.}
			\end{axis}
		\end{tikzpicture}
		\caption{Tensorized spline in $d=3$ (left) and $d=5$ (right) dimensions.}
		\label{fig:spline}
	\end{figure}
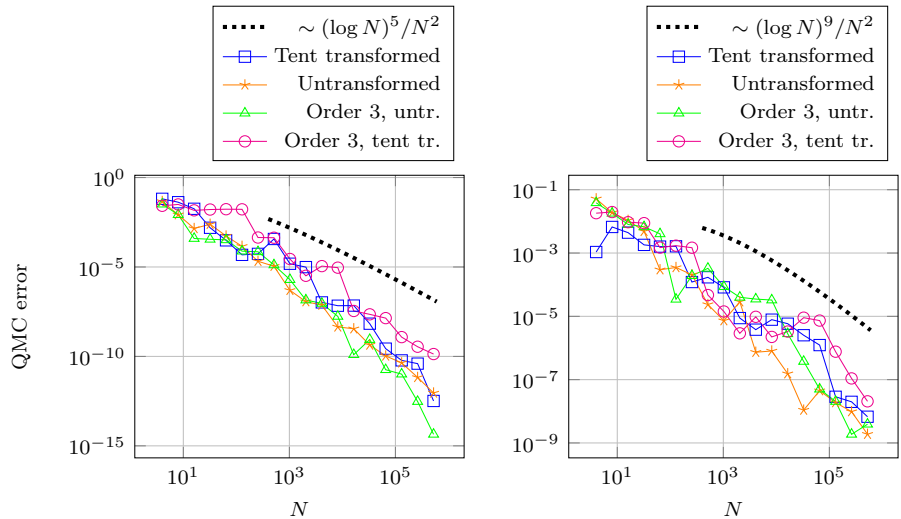 
	
	\subsection{Piecewise linear and smoother test functions}
	It is interesting to look at how tent transformed order $2$ nets perform for functions lying outside of the Sobolev space of dominating mixed smoothness $2$. For this we will take the tensor product of piecewise linear functions. We generate these functions by first choosing random nodes $0 = x_0 < x_1 < \dots < x_k = 1$ and corresponding values $y_i \in [-1, 1], i=0, 1, \dots, k$ uniformly at random. The univariate function $f$ is now given by the conditions $f(x_i) = y_i$ with linear interpolation between the nodes. We generate such a function for each coordinate and tensorize to obtain a $d$-variate function. Figure \ref{fig:piecewise_linear} shows the absolute error of the true integral and the quasi-Monte Carlo approximation using tent transformed order $2$ nets in $d=3$ dimensions (where we choose $3$, $4$ and $5$ intermediate nodes for each dimension respectively). The colored plots show $3$ random instances each for tent transformed and untransformed order $2$ nets. These tensorized, piecewise linear functions fulfill a similar condition to Theorem \ref{thm:FS_charact.}, namely
	\begin{equation}\label{eq:21}
		\|f\|_{\bB^{2,\text{dyad}}_{1,\infty}([0, 1]^d)}=\sup_{\bj \in \bbN_{-1}^d} 2^{|\bj|_1} \sum_{\bk \in \bbD_\bj} |d_{\bj, \bk}(f)| < \infty.
	\end{equation}
	With an analogous analysis as done in the proof of Theorem \ref{thm:upper_bound} (again using bounds from \cite{HMOU2016}) we obtain that for such functions the worst case error of the quasi-Monte Carlo rule for tent transformed order $2$ nets can be upper bounded by
	\begin{align*}
		\QMC(X, \bB^{2,\text{dyad}}_{1,\infty}([0, 1]^d)) \lesssim \frac{(\log N)^{2d-1}}{N^2}.
	\end{align*}
	
	We conclude with a look at functions with smoothness higher than $2$. We consider a cutout of a quadratic B-spline
	$$
	f(x) = \begin{cases}
		\frac34 - x^2 & , 0 \leq x \leq \frac12, \\
		\frac98 - \frac32 x + \frac12 x^2 & , \frac12 \leq x \leq 1,
	\end{cases}
	$$
	which was already used in \cite{Bar23, BLNU23} as a numerical example and has regularity $2.5 - \varepsilon$ on the scale of Sobolev-Hilbert spaces. We tensorize this function and look at the error of the quasi-Monte Carlo rule with tent transformed and untransformed order $2$ nets, see Figure \ref{fig:spline}. We also compare this to order $3$ nets (generated via digital interlacing with $3$ numbers). Unfortunately, from the plots it is still unclear if the added smoothness of these functions gets reflected in the error decay via tent transformed order $2$ nets.
	
	\paragraph{Acknowledgement.}
	
	BK and NN were supported by the ESF Plus Young Researchers Group ReSIDA-H2 (SAB, project number 100649391). TU acknowledges support by the German Research Foundation (DFG 403/4-1).
	
	\bibliographystyle{plain}

\end{document}